\documentclass[english]{article}
\usepackage[latin9]{inputenc}
\usepackage{amsmath}
\usepackage{amssymb}
\usepackage{graphicx}
\usepackage{esint}

\makeatletter

\newcommand{\lyxdot}{.}

\title{Newtheorem and theoremstyle test}
\author{Michael Downes\\updated by Barbara Beeton}

\usepackage[exercise]{amsthm}

\newtheorem{thm}{Theorem}[section]

\newtheorem{lem}[thm]{Lemma}

\theoremstyle{remark}
\newtheorem*{rmk}{Remark}

\theoremstyle{plain}

\newtheoremstyle{note}
  {3pt}
  {3pt}
  {}
  {}
  {\itshape}
  {:}
  {.5em}
  {}

\theoremstyle{note}

\newtheoremstyle{citing}
  {3pt}
  {3pt}
  {\itshape}
  {}
  {\bfseries}
  {.}
  {.5em}
  {\thmnote{#3}}

\theoremstyle{citing}

\newtheoremstyle{break}
  {9pt}
  {9pt}
  {\itshape}
  {}
  {\bfseries}
  {.}
  {\newline}
  {}

\theoremstyle{break}

\theoremstyle{exercise}

\swapnumbers
\theoremstyle{plain}

\let\lvert=|\let\rvert=|

\usepackage{babel}

\usepackage{babel}

\usepackage{babel}

\usepackage{babel}

\makeatother

\usepackage{babel}
\begin{document}

\title{Unstable CMC spheres and outlying CMC spheres in AF 3-manifolds}

\author{Shiguang Ma\thanks{School of Mathematical Sciences and LPMC, Nankai University, Tianjin
300071, People\textquoteright s Republic of China, msgdyx8741@nankai.edu.cn.The
author was supported by NSFC grant No.11301284, NSFG grant No. 11571185
and ``Specialized Research Fund for the Doctoral Program of Higher
Education'', No.20120031120028.}}
\maketitle
\begin{abstract}
In this paper, we introduce a non-linear ODE method to construct CMC
surfaces in Riemannian manifolds with symmetry. As an application
we construct unstable CMC spheres and outlying CMC spheres in asymptotically
Schwarzschild manifolds with metrics like $g_{ij}=(1+\frac{1}{l})^{2}\delta_{ij}+O(l^{-2})$.
The existence of unstable CMC spheres tells us that the stability
condition in Qing-Tian's work \cite{Qing-Tian-CMC} can not be removed
generally.

\end{abstract}

\section{Introduction}

Constant mean curvature (CMC) surfaces are a kind of important submanifolds.
In this paper we mainly focus on the CMC spheres in asymptotically
flat (AF) 3-manifolds. First let's state some background works in
this area. 

In 1996, G. Huisken and S.T. Yau proved the existence of a foliation
of stable CMC spheres in asymptotically Schwarzschid manifolds in
\cite{Huisken-Yau}. They also proved, under certain radius condition,
the uniqueness of the stable CMC spheres. In \cite{Qing-Tian-CMC},
J. Qing and G. Tian removed this radius condition, i.e. they proved
that, outside certain compact subset, any stable CMC sphere which
separates the compact part from infinity belongs to those constructed
in \cite{Huisken-Yau}. There are many sequential works. Lan-hsuan
Huang did similar work as Huisken and Yau in general AF manifolds
with RT conditions (those are a series of asymptotically ``odd''
or ``even'' conditions). See \cite{Huang-CMC}. In \cite{NERZ-CMC},
Nerz considered the existence of CMC foliation in AF manifolds without
RT conditions. The uniqueness results of Huang and Nerz also need
radius conditions. In \cite{Shiguang-Ma-CMC}, I proved a uniqueness
result under mild radius condition, which improved Huang's uniqueness
result in \cite{Huang-CMC} in some special case. In \cite{Shiguang-Ma-CMC2},
I removed the radius condition in proving the uniqueness in AF manifolds
with decay rate of the metic to be $-1$. Gang Tian and Andre Neves
have also considered similar existence and uniqueness problems in
asymptotically hyperbolic manifolds in \cite{Neves-Tian1} and \cite{Neves-Tian2}.

However, in all the above works, one has to assume stability of the
CMC surfaces to prove the uniqueness ( In Nerz's work he used an integral
estimate on the mean curvature instead ). It was asked by Professor
Gang Tian that whether one can remove the stability condition in proving
the uniqueness and if the answer is no generally, when we can do this.
In \cite{CMC-Warped-Metric} Simon Brendle prove that in a class of
warped product metrics the only CMC spheres are umbilic. In particular,
in Schwarzschild manifolds with positive mass, the only embedded CMC
surfaces are spheres of symmetry. This theorem requires assumptions
on neither the topology nor the stability of the surfaces. Then it
is natural to consider asymptotically Schwarzschild manifolds with
positive mass. In this paper without loss of generality we consider
asymptotically Schwarzschild manifolds with only one asymptotically
flat end. Suppose $(M,ds^{2})$ is a complete Riemannian manifold.
For a compact subset $K\subset M$, we assume $M\backslash K$ is
diffeomorphic to $\mathbb{R}^{3}\backslash\bar{B}_{1}(0)$ and we
assume $\{x_{i}\}_{i=1}^{3}$ are the standard coordinates of $\mathbb{R}^{3}.$
Denote
\[
l=\sqrt{x_{1}^{2}+x_{2}^{2}+x_{3}^{3}}.
\]
 Note that the Schwarzschild metric with mass $1$ takes the form
\[
g_{ij}^{S}=(1+\frac{1}{2l})^{4}\delta_{ij}
\]
with $\delta_{ij}$ to be the standard Euclidean metric. 

We use the cylindrical coordinates $(x,r,\theta)$ to calculate, namely,
\[
\begin{cases}
x_{1}=x,\\
x_{2}=r\cos\theta,\\
x_{3}=r\sin\theta,
\end{cases}
\]
where $r=\sqrt{x_{2}^{2}+x_{3}^{2}}.$

For $\lambda>0$ and $p\in\mathbb{R}$, let $\phi(s)=\phi_{\lambda,p}(s),s\in[0,+\infty]$
be a smooth monotonic function such that 
\begin{align}
\phi(s) & =\begin{cases}
1, & s\in[0,\lambda],\\
p, & s\in[2\lambda,+\infty],
\end{cases}\label{phi definition}
\end{align}
and 
\begin{equation}
\lambda|\phi'|+\lambda^{2}|\phi''|\leq C|p-1|.\label{phi basic}
\end{equation}

We consider asymptotically Schwarzschild metrics which take the following
form on $\mathbb{R}^{3}\backslash\bar{B}_{1}(0)$
\[
ds^{2}(\partial_{i},\partial_{j})=\delta_{ij}+T_{ij}
\]
where
\begin{equation}
\begin{cases}
T_{rr} & =g_{1}(r,x)=\frac{2}{l}+\frac{1}{l^{2}},\\
T_{\theta\theta} & =g_{2}(r,x)=r^{2}(\frac{2}{l}+\frac{1}{l^{2}}\phi_{\lambda,p}(|\frac{r}{x}|)),\\
T_{xx} & =g_{3}(r,x)=\frac{2}{l}+\frac{1}{l^{2}},\\
T_{rx} & =T_{r\theta}=T_{x\theta}=0.
\end{cases}\label{g1,g2,g3}
\end{equation}
Note that when $|\frac{r}{x}|=+\infty$, it should be understood as
that $x=0.$

Let 
\[
\chi(r,x)=\frac{1+\frac{2}{l}+\frac{1}{l^{2}}\phi_{\lambda,p}(|\frac{r}{x}|)}{1+\frac{2}{l}+\frac{1}{l^{2}}}.
\]
 So we have $r^{2}+T_{\theta\theta}=r^{2}\chi(r,x)(1+g_{1}(r,x)).$

It is obvious that for fixed $\lambda,p$, $ds^{2}$ is a smooth metric
and 
\[
ds^{2}=g^{S}+O(l^{-2}).
\]

\begin{rmk}Note that 
\[
\begin{cases}
x_{2} & =r\cos\theta,\\
x_{3} & =r\sin\theta.
\end{cases}
\]
Assume $x=x_{1}>0.$ Let $ds^{2}-(1+\frac{1}{l})^{2}\delta_{ij}=S_{ij}.$
Then by direct calculation we have 
\[
S_{ij,ij}-S_{ii,jj}=-\frac{\frac{2x\phi'(\frac{r}{x})}{r}+\phi''(\frac{r}{x})}{x^{4}}.
\]
For any $0<\alpha<\beta.$ 
\begin{align*}
 & \int_{\alpha\leq x_{1}\leq\beta}(S_{ij,ij}-S_{ii,jj})dx_{1}dx_{2}dx_{3}\\
= & -2\pi\int_{\alpha}^{\beta}\frac{1}{x^{2}}dx\int_{\lambda}^{2\lambda}(2\phi'(s)+s\phi''(s))ds\\
= & 2\pi(\frac{1}{\beta}-\frac{1}{\alpha})(p-1).
\end{align*}
So when $p>1$, $S_{ij,ij}-S_{ii,jj}$ makes a negative contribution
to the scalar curvature. Whether the scalar curvature of an asymptotically
Schwarzschild manifold is positive is a very delicate thing. When
the scalar curvature is negative, we have more flexibility while when
the scalar curvature is positive, we have more rigidity. This can
also be seen in \cite{OUTLYING-CMC} and \cite{Effective-PMT}.

\end{rmk}

Now we can state the main theorems of this paper.

\begin{thm} \label{thm 1} Suppose $(M\backslash K,ds^{2})$ is an
asymptotically flat end with $ds^{2}(\partial_{i},\partial_{j})=\delta_{ij}+T_{ij}$
with $T_{ij}$ given by (\ref{g1,g2,g3}). Then we can fix $\lambda>0$
small and $p>0$ large such that there is $\delta(\lambda,p)>0$ such
that for $0<H<\delta(\lambda,p)$, there is an embedded CMC sphere
$\Sigma$ which is unstable, has mean curvature $H$ and separates
$K$ from infinity. All the unstable CMC spheres constructed can be
parameterized by its mean curvature $H.$ If we denote the area of
$\Sigma(H)$ by $|\Sigma(H)|$ and $l_{0}=\inf_{x\in\Sigma}|x|$ then
we have 
\begin{align*}
l_{0}=O(H^{-1}),\\
|H^{2}|\Sigma(H)|-48\pi| & \leq C(p)H.
\end{align*}

\end{thm}

This theorem answers Tian's question partially. In the case of \cite{Qing-Tian-CMC}
(hence in the case of \cite{Shiguang-Ma-CMC2}), the stability condition
can not be removed. 

In a more recent work \cite{OUTLYING-CMC} of S. Brendle and M. Eichmair,
they constructed outlying CMC spheres which are stable in asymptotically
Schwarzschild manifold with positive mass. By ``outlying'' they
meant that the compact region bounded by $\Sigma$ is disjoint from
$B_{l_{0}(\Sigma)}(0)$. Our approach also works in such a setting. 

\begin{thm}\label{thm 2} There is an asymptotically Schwarzschild
manifold $(M,g)$ with $m=1$ in which we can construct two sequences
of embedded outlying CMC spheres which are parameterized by $\Sigma_{n}^{1},\Sigma_{n}^{2}$
and a sequence of CMC spheres $\Sigma_{n}^{3}$ which separate the
compact part from infinity, such that 
\[
H(\Sigma_{n}^{1})=H(\Sigma_{n}^{2})=H(\Sigma_{n}^{3})
\]
and $\Sigma_{n}^{1},\Sigma_{n}^{3},\Sigma_{n}^{2}$ share a common
symmetry axis and $\Sigma_{n}^{1},\Sigma_{n}^{2}$ are tangent to
$\Sigma_{n}^{3}$ at the two poles of $\Sigma_{n}^{3}$ where $x_{2}=x_{3}=0.$

\end{thm}

This result is a little weaker than that of \cite{OUTLYING-CMC} in
that we do not discuss the stability of the surfaces. 

The main technique of this paper is direct analysis of a nonlinear
ODE which comes from the study of Delaunay type CMC surfaces. In a
recent work \cite{Delaunay-CMC-Geodesic}, Frank Pacard and I proved
that along any non-degenerate closed embedded geodesic in a 3-manifold,
one can construct constant mean curvature surfaces of Delaunay type
of arbitrarily small size. It was asked by Frank Pacard that if we
can deform the neck size of a Delaunay type CMC surface along a closed
geodesic to develop a singularity (which can be regarded as the boundary
of the moduli of CMC surfaces). If one can do this and prove that
the singularity is removable in certain sense, then one can construct
CMC clusters in Riemannian manifolds. In certain sense, this paper
is also related to this problem. The unstable CMC spheres constructed
are in fact CMC clusters of ``three bubbles'' in asymptotically
Schwarzschild manifold. The author believes that this ODE method can
be used to construct many examples. Also we can use it to treat the
question raised by Pacard in the case that the metric has rotational
symmetry along the geodesic.

The main contributions can be summarized as follows: The meridian
curve of the CMC surfaces of revolution is determined by an ODE of
the following type
\[
\begin{cases}
g_{a}''(y) & -\frac{1}{g_{a}(y)}(1+g'_{a}(y)^{2})+(2+\rho)(1+g_{a}'(y)^{2})^{\frac{3}{2}}=0,\\
g_{a}(0) & =a\in[0.95,1.05],\\
g_{a}'(0) & =0.
\end{cases}
\]
where the expression of $\rho=\rho(H,g_{a},g_{a}',y)$ is given by
(\ref{rho-original}) and $|\rho|\leq CH$ holds, where $H$ is the
mean curvature, which is assumed to be a small positive number. If
we neglect $\rho$, when $a<1$, the solution is the meridian curve
of Delaunay unduloid in $\mathbb{R}^{3}.$ As $a\rightarrow1^{-}$
the solution develops singularities at each integer at the same time.
And the Delaunay surface converges to infinitely many spheres each
one of which meets its two neighbors at its two poles. If we do not
neglect $\rho$, for a fixed $a<1$, when $H$ is very small, the
solution looks similar to $\rho=0$ case. Look at the graph below.
The difference is that the solution is no longer exactly periodic.
When $a$ goes up, the solution $g_{a}(y),y\in[0,5]$ will develop
a singularity. The singularity will develop at $z_{1}$ or $z_{2}$.
To study this, we use the Delaunay parameter function
\[
\tau(g_{a}(y),g_{a}'(y))=-g_{a}^{2}(y)+\frac{g_{a}(y)}{\sqrt{1+g_{a}'(y)^{2}}}.
\]

\includegraphics[scale=0.4]{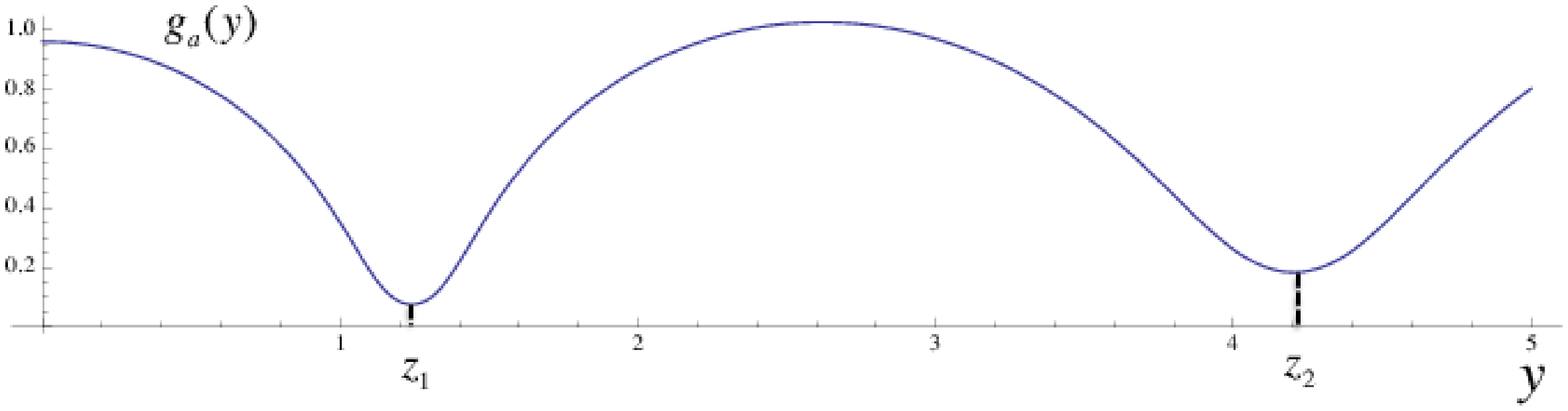}

Before a singularity develops, $\tau(g_{a}(z_{1}),g_{a}'(z_{1})),\tau(g_{a}(z_{2}),g_{a}'(z_{2}))>0.$
When a singularity develops, $g_{a}\rightarrow0,g_{a}'\rightarrow\infty$,
so $\tau\rightarrow0.$ So we will compare $\tau(g_{a}(z_{1}),g_{a}'(z_{1}))$
with $\tau(g_{a}(z_{2}),g_{a}'(z_{2})).$ Note that we have 
\[
\frac{d}{dy}\tau(g_{a}(y),g_{a}'(y))=g_{a}g_{a}'\rho.
\]
So
\[
\tau(g_{a}(z_{2}),g_{a}'(z_{2}))-\tau(g_{a}(z_{1}),g_{a}'(z_{1}))=\int_{z_{1}}^{z_{2}}g_{a}g_{a}'\rho dy.
\]
So we have to study the function $\rho.$ In our case, after some
more calculations and arguments, we know (\ref{rho-original}) is
reduced to
\[
\rho=\frac{H}{d}+\frac{H(g_{a}-yg_{a}')}{d^{3}\sqrt{1+g_{a}'^{2}}}+\frac{H^{2}(\phi-3)(g_{a}-yg_{a}')}{4d^{3}\sqrt{1+g_{a}'^{2}}}-\frac{H^{2}\phi'(y+g_{a}g'_{a})}{8d^{2}y^{2}\sqrt{1+g_{a}'^{2}}}+C(p,\lambda)O(H^{3})
\]
where $\phi$ is given by (\ref{phi definition}). Note that $\phi$
has two parameters $\lambda,p.$ At first glance, $\phi$ only appears
in $H^{2}$ terms, which is not dominant. Here our key observation
is that when $H$ is very small 
\[
\int_{z_{1}}^{z_{2}}g_{a}g_{a}'(\frac{H}{d}+\frac{H(g_{a}-yg_{a}')}{d^{3}\sqrt{1+g_{a}'^{2}}})dy=O(H^{2})
\]
 which is verified in Lemma \ref{key lemma}. So when $|p|$ is very
large, the third term becomes dominant. So we have freedom to choose
the place where the singularity develops, by choosing proper $p$
and $\lambda.$ Despite the singularity, the surface of revolution
is still smooth from Lemma \ref{regularity at singular point}. If
the singularity develops at $z_{2},$ we can get the ``three-bubble''
like CMC sphere by applying symmetric extension. Still, we can choose
$\lambda$ and $p$ properly such that the singularity develops simultaneously
at $z_{1}$ and $z_{2}$. Then we can construct outlying CMC spheres. 

The paper is organized as follows. In Section 2, we derive the ODE
which is satisfied by the meridian curve of a surface of revolution
with constant mean curvature. In Section 3, we analyze this ODE and
prove the existence of the solution. In Section 4, we deform the Delaunay
parameter until the singularity appears. And we can analyze the behavior
of the solution at the singular point. In Section 5, we prove the
two theorems.

\section{The mean curvature}

Suppose a surface of revolution $\Sigma\subset M$ is defined by 
\[
r=f(x)
\]
 where $f(x)$ is a smooth positive function. We are going to calculate
the mean curvature of such a surface.

We assume $(s,\theta)$ to be the coordinate on the surface which
satisfies 
\[
\frac{\partial}{\partial s}=\frac{\partial}{\partial x}+\frac{\partial f}{\partial x}\frac{\partial}{\partial r},\frac{\partial}{\partial\theta}=x_{2}\frac{\partial}{\partial x_{3}}-x_{3}\frac{\partial}{\partial x_{2}}.
\]
We have 
\[
\begin{cases}
g(\partial_{s},\partial_{s}) & =1+f'(x)^{2}+g_{3}+f'(x)^{2}g_{1},\\
g(\partial_{\theta},\partial_{\theta}) & =r^{2}+g_{2}.
\end{cases}
\]
The unit normal vector on $\Sigma$ is 
\[
V=V_{r}\partial_{r}+V_{x}\partial_{x}
\]
where
\[
\begin{cases}
V_{r} & =\frac{1}{\sqrt{1+g_{1}}\sqrt{1+\frac{1+g_{1}}{1+g_{3}}f'^{2}}},\\
V_{x} & =-f'\frac{\sqrt{1+g_{1}}}{(1+g_{3})\sqrt{1+\frac{1+g_{1}}{1+g_{3}}f'^{2}}}.
\end{cases}
\]
We use $V_{e}$ to represent the normal vector in Euclidean metric,
i.e.
\[
V_{e}=\frac{\partial_{r}-f'\partial_{x}}{\sqrt{1+f'^{2}}}.
\]

We use $<\cdot,\cdot>$ to represent the inner product in the metric
$g$ and ``$\cdot$'' to represent the inner product in the Euclidean
metric. The mean curvature 
\[
H(\Sigma)=g^{ss}<\nabla_{\partial_{s}}V,\partial_{s}>+g^{\theta\theta}<\nabla_{\partial_{\theta}}V,\partial_{\theta}>
\]
\begin{align*}
<\nabla_{\partial_{s}}V,\partial_{s}>= & <(\partial_{s}V_{r})\partial_{r},\partial_{s}>+<V_{r}\nabla_{\partial_{s}}\partial_{r},\partial_{s}>\\
 & +<(\partial_{s}V_{x})\partial_{x},\partial_{s}>+<V_{x}\nabla_{\partial_{s}}\partial_{x},\partial_{s}>.
\end{align*}
By direct calculations we have 
\begin{align*}
<(\partial_{s}V_{r})\partial_{r},\partial_{s}> & =(\partial_{s}V_{r})f'(1+g_{1}),\\
<(\partial_{s}V_{x})\partial_{x},\partial_{s}> & =(\partial_{s}V_{x})(1+g_{3}),\\
<V_{r}\nabla_{\partial_{s}}\partial_{r},\partial_{s}> & =\frac{V_{r}}{2}(\partial_{r}g_{xx}+f'^{2}\partial_{r}g_{rr}),\\
<V_{x}\nabla_{\partial_{s}}\partial_{x},\partial_{s}> & =\frac{V_{x}}{2}(\partial_{x}g_{xx}+f'^{2}\partial_{x}g_{rr}),
\end{align*}
 
\begin{align*}
<\nabla_{\partial_{\theta}}V,\partial_{\theta}> & =V_{r}<\nabla_{\partial_{\theta}}\partial_{r},\partial_{\theta}>+V_{x}<\nabla_{\partial_{\theta}}\partial_{x},\partial_{\theta}>\\
 & =\frac{1}{2}(V_{r}\partial_{r}g_{\theta\theta}+V_{x}\partial_{x}g_{\theta\theta}).
\end{align*}
So we get 
\begin{align*}
H(\Sigma) & =(1+f'(x)^{2}+g_{3}+f'(x)^{2}g_{1})^{-1}[(\partial_{s}V_{r})f'(1+g_{1})+(\partial_{s}V_{x})(1+g_{3})\\
 & +\frac{V_{r}}{2}(\partial_{r}g_{xx}+f'^{2}\partial_{r}g_{rr})+\frac{V_{x}}{2}(\partial_{x}g_{xx}+f'^{2}\partial_{x}g_{rr})]\\
 & +(r^{2}+g_{2})^{-1}\frac{1}{2}(V_{r}\partial_{r}g_{\theta\theta}+V_{x}\partial_{x}g_{\theta\theta}).
\end{align*}
From (\ref{g1,g2,g3}), the expression of $H(\Sigma)$ can be reduced
to 
\begin{align*}
H(\Sigma)= & \frac{1}{2}(1+g_{1})^{-\frac{3}{2}}(1+f'^{2})^{-\frac{3}{2}}(-2(1+g_{1})f''\\
 & +(1+f'^{2})(\partial_{r}g_{1}-f'\partial_{x}g_{1})\\
 & +(r^{2}+g_{2})^{-1}(1+g_{1})(1+f'^{2})(2r+\partial_{r}g_{2}-f'\partial_{x}g_{2}))\\
= & -(1+g_{1})^{-\frac{1}{2}}(1+f'^{2})^{-\frac{3}{2}}f''-2(V_{e}\cdot x')(l^{-3}+l^{-4})(1+g_{1})^{-\frac{3}{2}}\\
 & +(1+g_{1})^{-\frac{1}{2}}(1+f'^{2})^{-\frac{1}{2}}f^{-1}+\frac{1}{2}(1+g_{1})^{-\frac{1}{2}}\chi(r,x)^{-1}V_{e}(\chi).
\end{align*}
where $x'=(x,f(x)\cos\theta,f(x)\sin\theta)$ and $V_{e}(\chi)$ is
the derivative of $\chi(r,x)$ in the direction $V_{e}=\frac{\partial_{r}-f'\partial_{x}}{\sqrt{1+f'^{2}}}.$ 

Now we have 
\begin{equation}
\frac{d^{2}f}{dx^{2}}-\frac{1}{f}(1+(\frac{df}{dx})^{2})+(H+\tilde{\rho})(1+(\frac{df}{dx})^{2})^{\frac{3}{2}}=0,\label{Original ODE}
\end{equation}
where 
\begin{align*}
\tilde{\rho} & =\frac{H}{l}+2(V_{e}\cdot x')(l^{-3}+l^{-4})(1+g_{1})^{-1}-\frac{1}{2}\chi^{-1}V_{e}(\chi).
\end{align*}
 Since we want to get constant mean curvature surfaces, we assume
$H$ to be a small positive constant. Denote 
\[
g(y)=\frac{H}{2}f(\frac{2}{H}y)=\frac{H}{2}f(x).
\]
We have 
\[
g'(y)=f_{x}(\frac{2}{H}y),g''(y)=\frac{2}{H}f_{xx}(\frac{2}{H}y),
\]
So we have 
\begin{equation}
g''(y)-\frac{1}{g(y)}(1+g'(y)^{2})+(2+\rho)(1+g'(y)^{2})^{\frac{3}{2}}=0\label{ODE}
\end{equation}
where
\[
\rho=\frac{2}{H}\tilde{\rho}.
\]
We want to study ODE (\ref{ODE}) instead of (\ref{Original ODE}).

\section{The analysis of the ODE}

Consider
\begin{equation}
\begin{cases}
g''(y)-\frac{1}{g(y)}(1+g'(y)^{2})+(2+\rho)(1+g'(y)^{2})^{\frac{3}{2}}=0,\\
g(0)=a,\\
g'(0)=0.
\end{cases}\label{ODE with initial value}
\end{equation}
We assume $a\in[0.9,1.1]$ and denote the solution by $g_{a}(y).$
Let us assume that $H>0$ is very small and the real number $p$ is
fixed. Let $d=\frac{H}{2}l.$ Then we have
\begin{align}
\rho= & \frac{H}{d}+(\frac{H}{2}+\frac{H^{2}}{4d})\frac{g_{a}-yg_{a}'}{d^{3}(1+g_{1})\sqrt{1+g_{a}'^{2}}}\nonumber \\
 & +(\frac{H}{2}+\frac{H^{2}\phi}{4d})\frac{g_{a}-yg_{a}'}{d^{3}(1+\frac{H}{d}+\frac{\phi H^{2}}{4d^{2}})\sqrt{1+g_{a}'^{2}}}\label{rho-original}\\
 & -\frac{H^{2}}{8d^{2}}\frac{\phi'}{(1+\frac{H}{d}+\frac{\phi H^{2}}{4d^{2}})}\frac{y+g_{a}g'_{a}}{y^{2}\sqrt{1+g_{a}'^{2}}}.\nonumber 
\end{align}

In the following analysis, by a constant $C$ ( or $C_{i},i=1,2,\cdots$)
we mean a general (or a particular) uniform positive constant which
does not depend on $H$, $a$,$p,\lambda$ or $y$. We use symbols
$C(H),C(p)$ to denote constants which depend on $H,p,$ etc. 

First we have the local existence. Let $A_{1},A_{2},A_{3},A_{4}$
be four positive constant. Let's denote a domain in the phase space
\[
\{(g,g');A_{2}\leq g\leq A_{1},-A_{4}\leq g'\leq A_{3}\}
\]
 as $D(A_{1},A_{2},A_{3},A_{4}).$ 

\begin{lem}\label{local existence} Suppose $g_{a}(y),y\in[\alpha,\beta],0\leq\alpha\le\beta$
solves (\ref{ODE with initial value}). And for some $A_{i},i=1,\cdots,4$,
$(g_{a},g_{a}')\in D(A_{1},A_{2},A_{3},A_{4}),y\in[\alpha,\beta]$.
Then there is $\delta>0$ such that the solution $g_{a}$ can be extended
to $[\alpha,\beta+\delta).$ 

\end{lem}

\begin{proof}Consider a system that is equivalent to (\ref{ODE with initial value})
\[
\begin{cases}
g_{a}' & =w_{a},\\
w_{a}' & =\frac{1}{g_{a}}(1+w_{a}^{2})-(2+\rho)(1+w_{a}^{2})^{\frac{3}{2}}.
\end{cases}
\]
When $(g_{a},w_{a})\in D(A_{1},A_{2},A_{3},A_{4})$ and $y\in[\alpha,\beta]$,
we know $d$ has lower bound $A_{2}$. The right hand side has uniform
bound and uniform Lipschitz constant with respect to $g_{a},w_{a}$.
From the local existence of ODE system, we can prove this lemma.

\end{proof}

Then we have a ``closed'' property for the solution.

\begin{lem}\label{closeness}Suppose $g_{a}(y),y\in[\alpha,\beta),0\leq\alpha<\beta<+\infty$
solves (\ref{ODE with initial value}). And for some $A_{i}>0,i=1,\cdots,4$,
$(g_{a},g_{a}')\in D(A_{1},A_{2},A_{3},A_{4}),y\in[\alpha,\beta)$.
If 
\[
\lim_{y\rightarrow\beta^{-}}g_{a}(y),\lim_{y\rightarrow\beta^{-}}g_{a}'(y)
\]
 exist and are finite, then $g_{a}(y)$ (as a solution to (\ref{ODE with initial value}))
can be extended to $y=\beta.$ In particular, if $g_{a}(y)$ and $g_{a}'(y)$
are both monotonic and bounded, $g_{a}$ can be extended to $y=\beta.$

\end{lem}

\begin{proof}If for some $A_{i},i=1,\cdots,4$, $(g_{a},g_{a}')\in D(A_{1},A_{2},A_{3},A_{4}),y\in[\alpha,\beta)$
and 
\[
\lim_{y\rightarrow\beta^{-}}g_{a}(y),\lim_{y\rightarrow\beta^{-}}g_{a}'(y)
\]
 exist, then from (\ref{ODE with initial value}) and (\ref{rho-original})
we have 
\[
\lim_{y\rightarrow\beta^{-}}g_{a}''(y)
\]
 exists. By defining the values of $g_{a},g_{a}',g_{a}''$ on $\beta$
in the obvious way, we can prove this lemma.

\end{proof}

\begin{lem}\label{beta-alpha estimate}Suppose $g_{a}(y)>0,y\in I$
solves (\ref{ODE with initial value}) where $I$ is an interval $[\alpha,\beta]$
or $[\alpha,\beta)$, $0\leq\alpha<\beta$. In the case of $I=[\alpha,\beta)$,
$g_{a}(\beta)$ and $g_{a}'(\beta)$ should be understood as a limit
which is assumed to exist.

Then 
\begin{enumerate}
\item If $g_{a}(y)$ is monotonic in $y$ and $0<C^{-1}<|g_{a}'|<C$ when
$y\in I$, then
\[
C^{-1}|g_{a}(\beta)-g_{a}(\alpha)|\leq|\beta-\alpha|\le C|g_{a}(\beta)-g_{a}(\alpha)|.
\]

\item If $g_{a}'(y)$ is monotonic in $y$ and $k_{a}=g_{a}\sqrt{1+g_{a}'(y)^{2}}$
satisfies that $k_{a}-\frac{1}{2}$ does not change sign and
\[
C^{-1}\leq|k_{a}-\frac{1}{2}|\leq C,y\in I.
\]
And we assume that $|\rho|\leq CH.$ Then when $H$ is small we have
\[
C^{-1}\inf_{y\in I}g_{a}\leq\frac{|\beta-\alpha|}{|\arctan g_{a}'(\beta)-\arctan g_{a}'(\alpha)|}\leq C\sup_{y\in I}g_{a}.
\]

\end{enumerate}
\end{lem}

\begin{proof}We need only prove the case $I=[\alpha,\beta]$. The
case $I=[\alpha,\beta)$ follows by considering $[\alpha,\beta_{i}]$
with $\beta_{i}\rightarrow\beta.$

The first one follows form 
\[
|\beta-\alpha|=|\int_{g_{a}(\alpha)}^{g_{a}(\beta)}\frac{dg_{a}}{g_{a}'}|
\]
and the second one follows from the fact that $\rho$ is small and
\[
|\beta-\alpha|=|\int_{g_{a}'(\alpha)}^{g_{a}'(\beta)}\frac{dg_{a}'}{g_{a}''}|=|\int_{g_{a}'(\alpha)}^{g_{a}'(\beta)}\frac{g_{a}dg_{a}'}{(1+g_{a}'^{2})(1-(2+\rho)k_{a})}|.
\]

\end{proof}

Now we will prove that for the solution of (\ref{ODE with initial value}),
$d=\sqrt{y^{2}+g_{a}^{2}}$ has uniform positive lower bound independent
of $H$ and $a\in[0,9,1,1]$ . This is an important estimate. Only
with this estimate can we know that $|\rho|\leq CH.$

\begin{lem}\label{d lower bound}We can choose $\lambda>0$ small
and $\delta(p,\lambda)>0$ such that when $H<\delta(p,\lambda)$,
for all $a\in[0.9,1.1]$, as long as the solution of (\ref{ODE with initial value})
can be extended when $y>0,$

\[
d=\sqrt{g_{a}^{2}+y^{2}}\geq C>0,
\]
\begin{equation}
|\frac{1}{8d^{2}}\frac{\phi'}{(1+\frac{H}{d}+\frac{\phi H^{2}}{4d^{2}})}\frac{y+gg'}{y^{2}\sqrt{1+g_{a}'^{2}}}|\leq\frac{C|p-1|}{\lambda},\label{phi' estimate}
\end{equation}
\[
|\rho|\leq\hat{C}H.
\]

\end{lem}

\begin{proof}

First let us pretend that $\phi=p$ holds everywhere. So $\phi'=0$
and 
\begin{align*}
\rho= & \frac{H}{d}+\frac{H}{2}\frac{g_{a}-yg_{a}'}{d^{3}(1+g_{1})\sqrt{1+g_{a}'^{2}}}+\frac{H^{2}}{4}\frac{g_{a}-yg_{a}'}{d^{4}(1+g_{1})\sqrt{1+g_{a}'^{2}}}\\
 & +\frac{H}{2}\frac{g_{a}-yg_{a}'}{d^{3}(1+\frac{H}{d}+\frac{\phi H^{2}}{4d^{2}})\sqrt{1+g_{a}'^{2}}}+\frac{H^{2}}{4}\frac{p(g_{a}-yg_{a}')}{d^{4}(1+\frac{H}{d}+\frac{\phi H^{2}}{4d^{2}})\sqrt{1+g_{a}'^{2}}}.
\end{align*}

Let $D_{a}(y)$ solves
\begin{equation}
\begin{cases}
D''(y)-\frac{1}{D(y)}(1+D'(y)^{2})+2(1+D'(y)^{2})^{\frac{3}{2}}=0,\\
D(0)=a,\\
D'(0)=0,
\end{cases}\label{Delaunay ODE}
\end{equation}
where $a\in[0.8,1.2]$. 

We state some facts of $D_{a}(y)$ without proof. There are $\delta_{0},C_{1},C_{2},C_{3}>0$
such that, 
\begin{enumerate}
\item For each $a\in[0.8,1.2],$ the solution of (\ref{Delaunay ODE}) exists
on $[0,\delta_{0}]$.
\item $R_{1}=\{(D_{a}(y),\frac{d}{dy}D_{a}(y));a\in[0.8,1.2],y\in[0,\delta_{0}]\}$
is trapped in the region 
\[
R_{2}=\{(\xi,\eta);C_{1}\leq\xi\le1.2,-C_{2}\leq\eta\leq0\}.
\]

\item When $a\in[0.8,1.2]$, the solution $D(y),y\in[0,\delta_{0}]$ satisfies
$D\sqrt{1+D'^{2}}\geq\frac{2}{3}.$
\item For each $a\in[0.8,1.2]$, $D_{a}(0)-D_{a}(\delta_{0})>C_{3}$ and
$D_{a}'(\delta_{0})<-C_{3}$.
\end{enumerate}
We may draw a graph in the phase space to illustrate this. In the
graph below, $\delta_{0}=0.75.$ We can choose $C_{1}=0.5,C_{2}=2.4.$
$D_{a}(0)-D_{a}(\delta_{0})$ and $-D'_{a}(\delta_{0})$ are positive
continuous functions of $a\in[0,8,1.2]$. So $C_{3}$ can be chosen
properly small and positive. 

\includegraphics[scale=0.4]{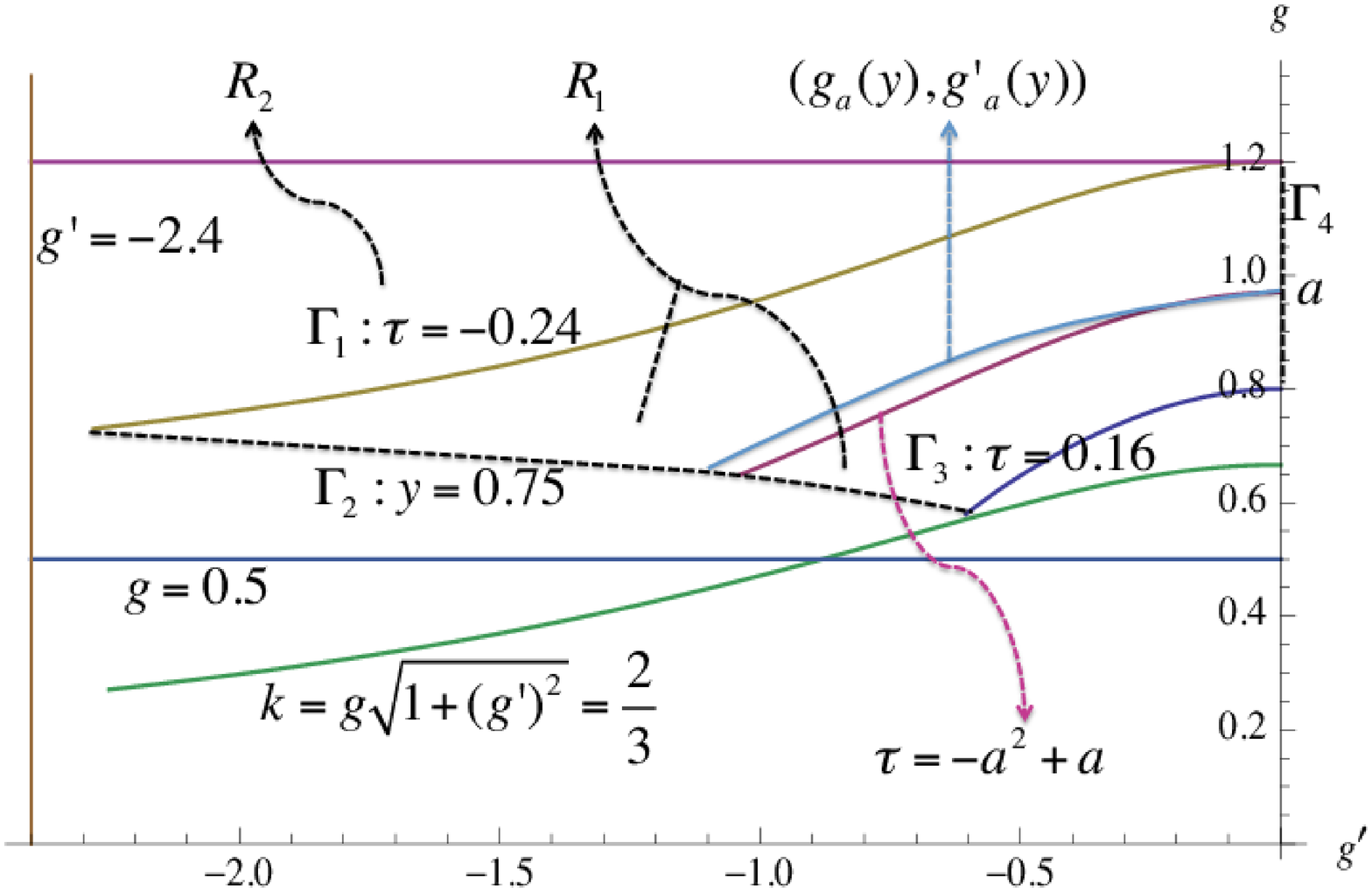}

Let $R_{1}$ be the closure of a bounded domain whose boundary contains
4 components, $\Gamma_{1},\Gamma_{2},\Gamma_{3},\Gamma_{4}.$
\begin{align*}
\Gamma_{1} & =\{(D_{a}(y),\frac{d}{dy}D_{a}(y));a=1.2,y\in[0,\delta_{0}]\},\\
\Gamma_{2} & =\{(D_{a}(y),\frac{d}{dy}D_{a}(y));a\in[0.8,1.2],y=\delta_{0}\},\\
\Gamma_{3} & =\{(D_{a}(y),\frac{d}{dy}D_{a}(y));a=0.8,y\in[0,\delta_{0}]\},\\
\Gamma_{4} & =\{(D_{a}(y),\frac{d}{dy}D_{a}(y));a\in[0.8,1.2],y=0\}.
\end{align*}
Now we consider the orbit $(g_{a},g_{a}')$ of (\ref{ODE with initial value}).
For each $a\in[0,9,1.1]$, $(g_{a}(0),g_{a}'(0))\in\Gamma_{4}.$ From
Lemma \ref{local existence}, the solution of (\ref{ODE with initial value})
can be extended a little to $y>0.$ Then the orbit will be extended
into the interior of $R_{1}.$ As long as the orbit keeps in $R_{1}$,
we know $d\geq C_{1}$. So there are $\delta(p)>0,C>0$ such that
when $0<H<\delta(p)$, 
\begin{equation}
|\rho|\leq CH.\label{rho estimate}
\end{equation}

If $\{(g_{a}(y),g_{a}'(y));y\in[0,\delta_{0}]\}\subset R_{1}$, we
have $d=\sqrt{y^{2}+d^{2}}\geq\min\{\delta_{0},C_{1}\}$ as long as
the solution exists. If it is not the case then we assume $y'\in(0,\delta_{0})$
is the first time $(g_{a},g_{a}')$ touches $\partial R_{1}$. We
claim that for $H$ sufficiently small, $(g_{a}(y'),g_{a}'(y'))\in\Gamma_{2}.$
For this we consider the Delaunay parameter function
\[
\tau(g_{a},g_{a}')=-g_{a}^{2}+\frac{g_{a}}{\sqrt{1+g_{a}'^{2}}}.
\]
 One can regard it as the first integral of (\ref{Delaunay ODE}).
When $y\in[0,y']$ 
\begin{equation}
\frac{d}{dy}\tau=g_{a}g'_{a}\rho.\label{tau-ODE}
\end{equation}
Note that $|g_{a}g_{a}'|\leq1.2C_{2}$. From (\ref{rho estimate}),
by choosing $H$ even smaller, we have for each $a\in[0,9,1.1]$,
as long as $(g_{a},g_{a}')\in R_{1},$ $(g_{a},g_{a}')$ keeps in
a small neighborhood of $(D_{a},D_{a}').$ So $(g_{a},g'_{a})$ has
no chance to touch $\Gamma_{1}$ or $\Gamma_{3}$. From
\begin{equation}
g_{a}\sqrt{1+g_{a}'^{2}}\geq\frac{2}{3}\label{k>2/3}
\end{equation}
and the fact that $\rho$ is small, we can deduce
\begin{equation}
g_{a}''(y)=\frac{1}{g_{a}(y)}(1+g_{a}'(y)^{2})(1-(2+\rho)g_{a}(y)(1+g_{a}'(y)^{2})^{\frac{1}{2}})\label{g'' and k}
\end{equation}
 is always negative, which implies both $g_{a}$ and $g_{a}'$ are
monotonically decreasing when $y\in[0,y']$ . From Lemma \ref{beta-alpha estimate},
we have 
\[
C(C_{1},C_{2},C_{3})^{-1}\leq y'\leq C(C_{1},C_{2},C_{3}).
\]
So if we choose $2\lambda\leq C(C_{1},C_{2},C_{3})^{-1}C_{1}$ and
define $\phi$ as (\ref{phi definition}), it will have no influence
on the analysis above. And we know 
\begin{align*}
d & \geq\min\{C(C_{1},C_{2},C_{3})^{-1},C_{1},\delta_{0}\}.
\end{align*}
Note that when $\phi'\neq0,$ there hold $y\geq C(C_{1},C_{2},C_{3})^{-1}$
and $|\frac{g}{y}|\leq2\lambda$. So from (\ref{phi basic}) 
\begin{align*}
 & |\frac{1}{8d^{2}}\frac{\phi'}{(1+\frac{H}{d}+\frac{\phi H^{2}}{4d^{2}})}(1+g'^{2})^{-\frac{1}{2}}\frac{y+gg'}{y^{2}}|\\
\leq & \frac{C|p-1|}{\lambda}.
\end{align*}
So from 
\[
\frac{|g_{a}-yg_{a}'|}{\sqrt{1+g_{a}'^{2}}\sqrt{y^{2}+g_{a}^{2}}}\leq1
\]
we can choose $\delta(p,\lambda)$ small such that $|\rho|\leq\hat{C}H.$

\end{proof}

Now we can expand $\rho$ in terms of $H.$ For $0<H<\delta(p,\lambda)$,
\[
\rho=\frac{H}{d}+\frac{H(g_{a}-yg_{a}')}{d^{3}\sqrt{1+g_{a}'^{2}}}+\frac{H^{2}(\phi-3)(g_{a}-yg_{a}')}{4d^{3}\sqrt{1+g_{a}'^{2}}}-\frac{H^{2}\phi'(y+g_{a}g'_{a})}{8d^{2}y^{2}\sqrt{1+g_{a}'^{2}}}+C(p,\lambda)O(H^{3}).
\]
From now on we will always assume that $\delta(p,\lambda)$ is so
small that $|\rho|\leq\hat{C}H.$

It is obviously that exactly one of the following three will happen
to $g_{a}(y)$,

(H1): The solution $g_{a}(y)$ can be extended at least two times
such that $g_{a}(y)\in(0,\frac{1}{2}),g_{a}'(y)=0$;

(H2): The solution $g_{a}(y)$ can be extended only one time such
that $g_{a}(y)\in(0,\frac{1}{2}),g_{a}'(y)=0$;

(H3): The solution $g_{a}(y)$ cannot be extended to a point $y$
such that $g_{a}(y)\in(0,\frac{1}{2}),g_{a}'(y)=0.$

Let's draw some pictures of the phase space to illustrate the three
cases. 

\includegraphics[scale=0.4]{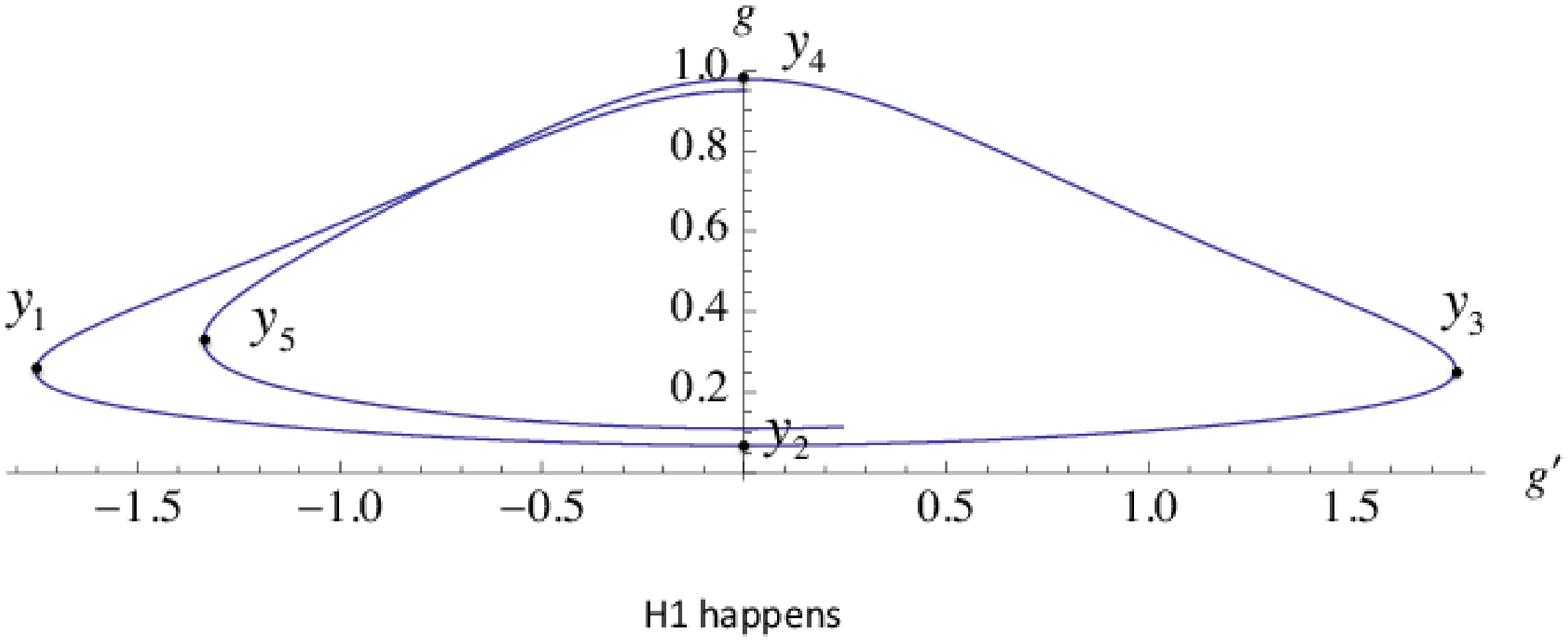}

\includegraphics[scale=0.4]{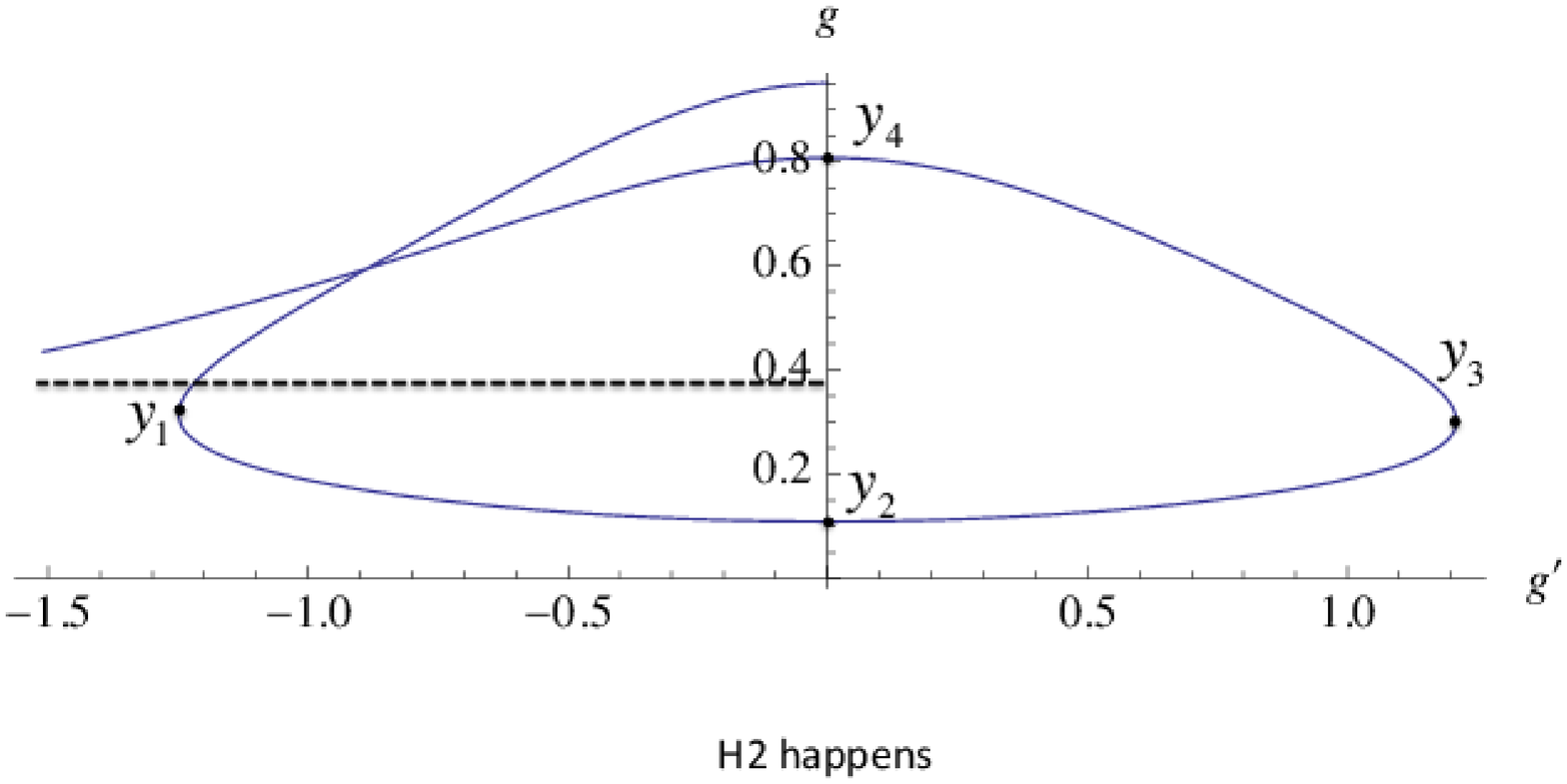}

\includegraphics[scale=0.4]{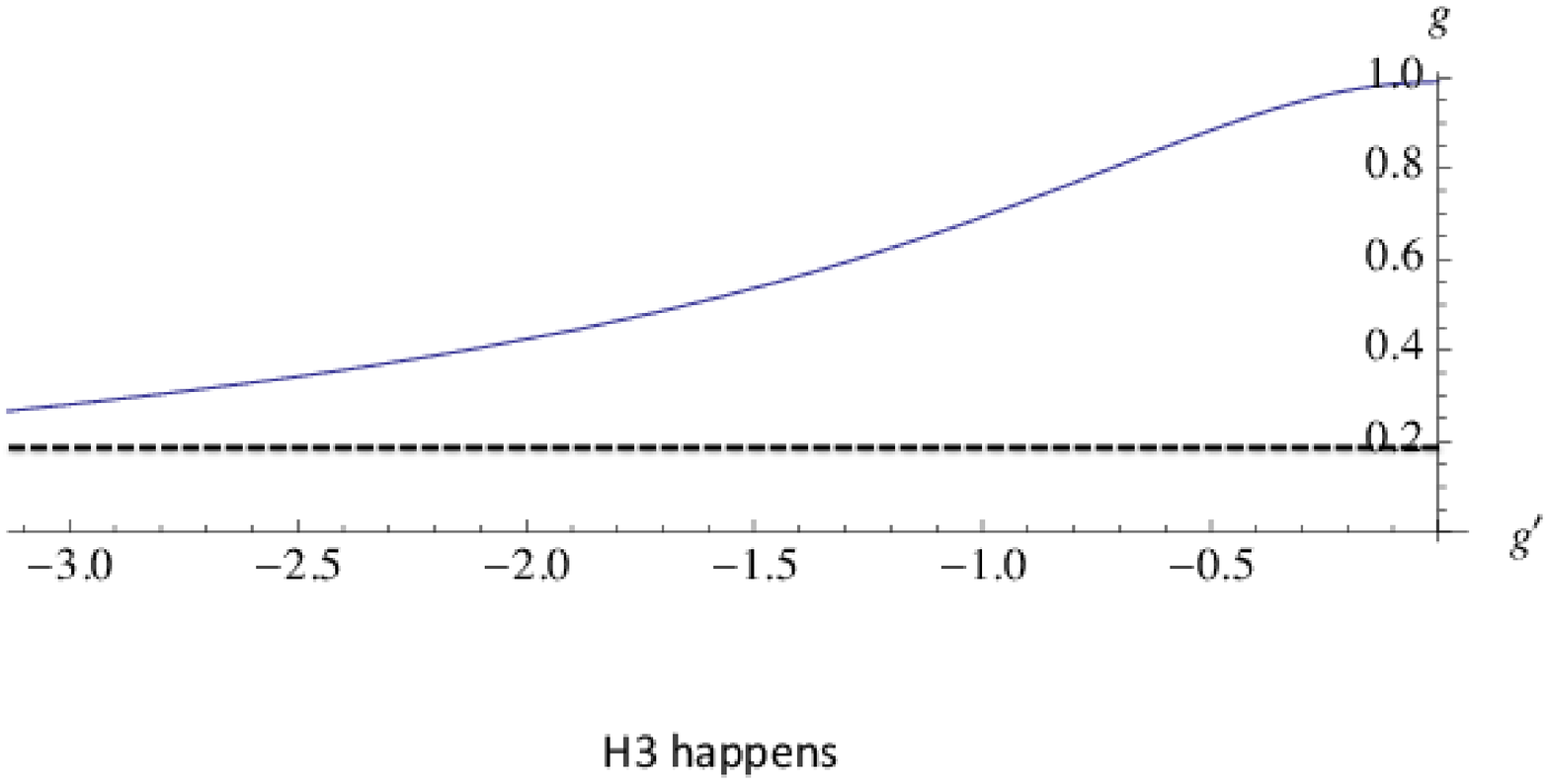}

\begin{lem}\label{y1 implies y2} There is a uniform $C_{4}>0$ such
that if $g_{a}(y)$ can be extended to some $y_{1}>0$ such that $0<g_{a}(y_{1})<C_{4},g'_{a}(y_{1})<0$
and $g''_{a}(y_{1})=0,$ then there is $y_{2}>y_{1}$ such that $g_{a}(y_{2})\in(0,\frac{1}{2})$
and $g'_{a}(y_{2})=0.$ Moreover, in $[y_{1},y_{2}]$, $g_{a}(y)$
is monotonically decreasing and $g'_{a}(y)$ is monotonically increasing.
(This can be seen in the case ``H1'' and ``H2''.)

\end{lem}

\begin{proof} From Lemma \ref{local existence}, we know that the
solution can be extended beyond $y_{1}$ for a small interval. $ $We
want to show that, when $y>y_{1}$, as long as the solution exists
and $g_{a}>0$ and $g_{a}'<0$, we have $g''_{a}>0.$ As before we
denote 
\[
k_{a}(y)=g_{a}\sqrt{1+g_{a}'^{2}}.
\]
Note that $|\rho|\leq\hat{C}H.$ If $k_{a}(y)<\frac{1}{2+\hat{C}H},$
we have
\[
g_{a}''(y)=\frac{1}{g_{a}}(1+g_{a}'^{2})(1-(2+\rho)k_{a})>0.
\]
From easy calculation 
\begin{equation}
\frac{dk_{a}(y)}{dy}=g_{a}'(1+g_{a}'^{2})^{\frac{1}{2}}+g_{a}(1+g_{a}'^{2})^{-\frac{1}{2}}g_{a}'g_{a}''.\label{ka-derivative}
\end{equation}
When $g_{a}>0,g_{a}'<0,$ we have
\begin{align*}
\frac{dk_{a}}{dg_{a}} & =\frac{dk_{a}(y)}{dy}\frac{1}{g_{a}'}=(1+g_{a}'^{2})^{\frac{1}{2}}+g_{a}(1+g_{a}'^{2})^{-\frac{1}{2}}g_{a}''\\
 & =\frac{k_{a}(2-(2+\rho)k_{a})}{g_{a}}.
\end{align*}
So we have 
\begin{equation}
\frac{dk_{a}}{k_{a}(2-(2+\rho)k_{a})}=\frac{dg_{a}}{g_{a}}.\label{Ka and ga}
\end{equation}
When $y=y_{1},$ from $g_{a}''(y_{1})=0,$ we have $2-(2+\rho(y_{1}))k_{a}(y_{1})=1$.
So as long as the solution can be extended when $y>y_{1}$ and $g_{a}>0,g_{a}'<0$
we have $k_{a}$ is monotonically decreasing. So once $k_{a}(y)<\frac{1}{2+\hat{C}H},$
$g_{a}''(y)>0$ holds as long as $y>y_{1},g_{a}>0,g_{a}'<0$.

The problem is whether we still have $g_{a}''(y)>0,$ when $y>y_{1}$
and $\frac{1}{2+\rho(y_{1})}\geq k_{a}\geq\frac{1}{2+\hat{C}H}$.
We choose $0<C_{4}<\frac{1}{\sqrt{2}(2+\hat{C}H)}$. When $0<g_{a}(y)\leq g_{a}(y_{1})\leq C_{4}$,
$k_{a}=g_{a}(y)\sqrt{1+g_{a}'^{2}(y)}\geq\frac{1}{2+\hat{C}H}$ and
$g_{a}'(y)<0$ we have 
\[
\sqrt{1+g_{a}'^{2}(y)}\geq\sqrt{2}
\]
which implies $g_{a}'(y)\leq-1.$ 

From direct calculations, $\exists$ a uniform constant $C>0$, such
that when $H$ is sufficiently small
\begin{align}
|\rho'| & \leq(\frac{H}{d^{3}}y+\frac{H}{d^{3}}g_{a}|g_{a}'|)(1+\frac{|g''_{a}|}{(1+(g'_{a})^{2})^{\frac{3}{2}}})+C(p,\lambda)H^{2}(1+|g_{a}'|)\nonumber \\
 & \leq CH(1+|g_{a}'|).\label{rho'}
\end{align}
From 
\begin{align*}
\frac{d}{dy}((2+\rho)k_{a}) & =\rho'k_{a}+(2+\rho)k_{a}'
\end{align*}
and (\ref{ka-derivative})(\ref{rho'}) we have 
\begin{equation}
\frac{d}{dg_{a}}((2+\rho)k_{a})\geq\frac{1}{2}\frac{k_{a}}{g_{a}}(2+\rho-|\frac{2CH(1+g_{a}')g_{a}}{g_{a}'}|).\label{(2+rho)ka}
\end{equation}
When $y>y_{1},k_{a}\geq\frac{1}{2+\hat{C}H}$, from 
\[
g_{a}'\leq-1
\]
 we have, for $H$ sufficiently small, $\frac{d}{dg_{a}}((2+\rho)k_{a})$
is positive. (Here we need the estimate $g_{a}'\leq-1$. We use $C_{4}$
for this technical reason.) So when $y>y_{1}$ and $\frac{1}{2+\rho(y_{1})}\geq k_{a}\geq\frac{1}{2+\hat{C}H},$
we have $(2+\rho)k_{a}$ monotonically decreasing as $y$ increases.
So when $y>y_{1},$
\[
g_{a}''(y)=\frac{1}{g_{a}}(1+g_{a}'^{2})(1-(2+\rho)k_{a})>0.
\]

By now we have known, when $y>y_{1}$, as long as the solution exists
and $0<g_{a}\leq C_{4}$ and $g_{a}'<0$, $g_{a}$ is monotonically
decreasing and $g_{a}'$ is monotonically increasing. Now we prove
that there is a positive lower bound for $g_{a}.$ We assume the solution
can be extended from $y_{1}$ to $\tilde{y}_{2}$ with $g_{a}(\tilde{y}_{2})>0$
and $g_{a}'(\tilde{y}_{2})<0.$ From (\ref{Ka and ga}) we know when
$y\in[y_{1},\tilde{y}_{2}]$ 
\[
\int_{g_{a}(y)}^{g_{a}(y_{1})}\frac{dg_{a}}{g_{a}}=\int_{k_{a}(y)}^{k_{a}(y_{1})}\frac{dk_{a}}{k_{a}(2-(2+\rho)k_{a})}\leq\int_{k_{a}(y)}^{k_{a}(y_{1})}\frac{dk_{a}}{k_{a}(2-(2+\hat{C}H)k_{a})}.
\]
So we have 
\[
\ln\frac{g_{a}(y_{1})}{g_{a}(y)}\leq(\frac{1}{2}\ln\frac{(2+\hat{C}H)k_{a}}{2-(2+\hat{C}H)k_{a}})|_{k_{a}(y)}^{k_{a}(y_{1})},
\]
which implies
\[
k_{a}(y)\leq C(\frac{g_{a}(y)}{g_{a}(y_{1})})^{2}.
\]
From 
\[
\frac{dg_{a}'}{dg_{a}}=\frac{g_{a}''}{g_{a}'}=\frac{(1+(g_{a}')^{2})[1-(2+\rho)k_{a}]}{g_{a}g_{a}'}
\]
we have
\begin{align*}
\int_{g_{a}'(\tilde{y}_{2})}^{g_{a}'(y_{1})}\frac{g_{a}'dg_{a}'}{1+(g_{a}')^{2}} & =\int_{g_{a}(\tilde{y}_{2})}^{g_{a}(y_{1})}\frac{1-(2+\rho)k_{a}}{g_{a}}dg_{a}\\
 & \geq\int_{g_{a}(\tilde{y}_{2})}^{g_{a}(y_{1})}\frac{1-C(\frac{g_{a}(y)}{g_{a}(y_{1})})^{2}}{g_{a}(y)}dg_{a}(y)
\end{align*}
which implies
\[
\frac{1}{2}\ln(1+g_{a}'^{2})|_{g_{a}'(\tilde{y}_{2})}^{g_{a}'(y_{1})}\geq\ln\frac{g_{a}(y_{1})}{g_{a}(\tilde{y}_{2})}-\frac{C}{2}g_{a}(y_{1})^{-2}(g_{a}(y_{1})^{2}-g_{a}(\tilde{y}_{2})^{2}).
\]
So we have 
\[
g_{a}(\tilde{y}_{2})\geq{\rm e}^{-C}(1+g_{a}'(y_{1})^{2})^{-\frac{1}{2}}g_{a}(y_{1}).
\]
Suppose $y_{2}$ is the supremum of the values $\tilde{y}_{2}$ until
which the solution can be extended and $g_{a}(\tilde{y}_{2})>0,g'_{a}(\tilde{y}_{2})<0.$
From Lemma \ref{beta-alpha estimate}, we can prove that $\tilde{y}_{2}$
is bounded (We can divide $g_{a},y\in[y_{1},\tilde{y}_{2}]$ into
two parts. On the first part $g_{a}$ is bounded, $g_{a}'<-\frac{1}{2}$
and on the second part $-\frac{1}{2}<g_{a}'<0$ and $k_{a}$ is away
from $\frac{1}{2}$ ), hence $y_{2}$ is finite. 

From Lemma \ref{closeness} and the monotonicity of $g_{a}$ and $g_{a}'$
we know $g_{a}$ can be extended to $y_{2}$ and 
\begin{align*}
g_{a}(y_{2}) & \geq{\rm e}^{-C}(1+g_{a}'(y_{1})^{2})^{-\frac{1}{2}}g_{a}(y_{1}),\\
g_{a}'(y_{2}) & \leq0.
\end{align*}
So we must have $g_{a}'(y_{2})=0,$ or there would be a contradiction
with the fact that $y_{2}$ is the supremum of $\tilde{y}_{2}$. Because
$C_{4}<\frac{1}{2}$, we have $0<g_{a}(y_{2})<\frac{1}{2}.$ 

\end{proof}

\begin{lem}\label{H3 holds} When $a\in[0.9,1.1]$, if (H3) happens,
then $\exists y_{1}'\in(0,+\infty)$ such that $g_{a}(y)$ can be
extended until $(0,y_{1}')$ but not until $(0,y_{1}']$. Moreover,
in $(0,y_{1}')$, $g_{a}(y)$ and $g_{a}'(y)$ are monotonically decreasing
and 
\[
\lim_{y\rightarrow y_{1}'^{-}}g_{a}(y)\geq0,\lim_{y\rightarrow y_{1}'^{-}}g_{a}'(y)=-\infty.
\]

\end{lem}

\begin{proof} When $a\in[0.9,1.1],$ we claim, if (H3) happens, when
$y>0$, as long as the solution can be extended, $g_{a}(y)>0,g_{a}'(y)<0,g_{a}''(y)<0$. 

First we prove this claim. When $y\in(0,\delta)$ for a small $\delta$,
$g_{a}(y)>0,g_{a}'(y)<0,g_{a}''(y)<0$ hold. If the claim were false
then there would be some $y_{1}>0$ such that 
\[
g_{a}(y_{1})>0,g_{a}'(y_{1})<0,g_{a}''(y_{1})=0
\]
and in $(0,y_{1})$, $g_{a}(y)>0,g_{a}'(y)<0,g_{a}''(y)<0$ hold.
From the monotonicity of $g_{a}$ we have 
\begin{equation}
\frac{d\tau}{dg_{a}}=\frac{d\tau}{dy}\frac{1}{g_{a}'}=g_{a}\rho.\label{tau w.r.t. g}
\end{equation}
Since $g_{a}$ is bounded, there is $C>0$ such that 
\begin{equation}
|\tau(g_{a}(y_{1}),g_{a}'(y_{1}))-(a-a^{2})|=|\tau(g_{a}(y_{1}),g_{a}'(y_{1}))-\tau(a,0)|\leq CH\label{tau(y1)-tau(0)}
\end{equation}

\[
g_{a}(y_{1})=\sqrt{\frac{\tau(g_{a}(y_{1}),g_{a}'(y_{1}))}{\frac{1}{k_{a}(y_{1})}-1}}=\sqrt{\frac{a-a^{2}\pm CH}{\frac{1}{k_{a}(y_{1})}-1}}.
\]
Note that $ $$k_{a}(y_{1})\in[\frac{1}{2+\hat{C}H},\frac{1}{2-\hat{C}H}]$.
So we can choose $H$ sufficiently small such that $g_{a}(y_{1})\leq\sqrt{\frac{0.9-0.9^{2}+CH}{1-CH}}\approx0.3.$
One can choose $C_{4}$ in Lemma \ref{y1 implies y2} such that $\sqrt{\frac{0.9-0.9^{2}+CH}{1-CH}}\leq C_{4}\leq\frac{1}{2(\sqrt{2}+CH)}$.
So $g_{a}(y_{1})\leq C_{4}.$ From Lemma \ref{y1 implies y2} one
can find a contradiction with (H3). So the claim is true.

Now we prove the lemma. Suppose the solution can be extended to $[0,y^{*}]$.
From $g_{a}(y^{*})>0,g_{a}'(y^{*})<0,g''_{a}(y^{*})<0$ and Lemma
\ref{local existence}, the solution can be extended beyond $y^{*}.$
So we know the maximal interval in which the solution exists should
look like $[0,y_{1}')$ with $y_{1}'$ possibly be $+\infty.$ To
see that $y_{1}'$ is finite, we just prove that when $-\frac{1}{2}<g_{a}'<0$,
$k_{a}$ is away from $\frac{1}{2}$. Then we can apply Lemma \ref{beta-alpha estimate},
by dividing the orbit into two parts. Note that 
\[
k_{a}=g\sqrt{1+g'^{2}}=\frac{1\pm\sqrt{1-4\tau(1+g'^{2})}}{2}.
\]
So we have 
\[
|k_{a}-\frac{1}{2}|=\frac{\sqrt{1-4\tau(1+g'^{2})}}{2}.
\]
From (\ref{tau(y1)-tau(0)}) and $g'^{2}\leq\frac{1}{4}$ we have
\[
|k_{a}-\frac{1}{2}|\geq\frac{\sqrt{0.55\pm CH}}{2}.
\]

And from the monotonicity we know $\lim_{y\rightarrow y_{1}'^{-}}g_{a}(y),\lim_{y\rightarrow y_{1}'^{-}}g'_{a}(y)$
exist ( $\lim_{y\rightarrow y_{1}'^{-}}g'_{a}(y)$ is possibly $-\infty$)
and 
\[
\lim_{y\rightarrow y_{1}'^{-}}g_{a}(y)\geq0.
\]
We will see $\lim_{y\rightarrow y_{1}'^{-}}g_{a}'(y)$ can not be
finite. Should it be finite, if $\lim_{y\rightarrow y_{1}'^{-}}g_{a}(y)=0$,
from (\ref{g'' and k}) one could find $y_{1}<y_{1}'$ such that $g''_{a}(y_{1})=0$
(where $k_{a}\approx\frac{1}{2}$) which is a contradiction. If $\lim_{y\rightarrow y_{1}'^{-}}g_{a}'(y)$
should be finite and $\lim_{y\rightarrow y_{1}'^{-}}g_{a}(y)>0$,
as $y_{1}'$ is finite, from Lemma \ref{closeness} and Lemma \ref{local existence},
the solution can be extended beyond $y_{1}$ which is also a contradiction.
So we have 
\[
\lim_{y\rightarrow y_{1}^{-}}g'_{a}(y)=-\infty.
\]

\end{proof} 

\begin{lem}\label{y2 to y3} If (H1) or (H2) happens, there exists
$y_{3}<+\infty$ such that $g_{a}$ can be extended to $y_{3}$ and
$g_{a}(y_{3})>g_{a}(y_{2}),g'_{a}(y_{3})>0,g''_{a}(y_{3})=0.$ In
$[y_{2},y_{3}]$ both $g_{a}$ and $g'_{a}$ are monotonically increasing.
Moreover, there is $C>0$ which does not depend on $a$ such that
\[
\tau(g_{a}(y_{3}),g_{a}'(y_{3}))\in[(1-CH)\tau(g_{a}(y_{1}),g_{a}'(y_{1})),(1+CH)\tau(g_{a}(y_{1}),g_{a}'(y_{1}))].
\]

\end{lem}

\begin{proof}So from (\ref{tau w.r.t. g}) we know
\begin{equation}
\frac{d\tau}{dg_{a}^{2}}=\frac{1}{2}\rho.\label{tau w.r.t. g^2}
\end{equation}
From $|\rho|\leq\hat{C}H$, we have 
\begin{equation}
|\tau(g_{a}(y_{2}),g_{a}'(y_{2}))-\tau(g_{a}(y_{1}),g_{a}'(y_{1}))|\leq CHg_{a}(y_{1})^{2}.\label{tau(y2)-tau(y1)}
\end{equation}
We know 
\[
\tau(g_{a}(y_{1}),g_{a}'(y_{1}))=(\frac{1}{k_{a}(y_{1})}-1)g_{a}(y_{1})^{2}\in[(1-CH)g_{a}(y_{1})^{2},(1+CH)g_{a}(y_{1})^{2}].
\]
 So 
\begin{equation}
\tau(g_{a}(y_{2}),g_{a}'(y_{2})))\in[(1-CH)g_{a}(y_{1})^{2},(1+CH)g_{a}(y_{1})^{2}].\label{tau(y2)}
\end{equation}
From $g_{a}(y_{2})\in(0,\frac{1}{2})$ and $g_{a}'(y_{2})=0$ we can
get finer estimate for $g_{a}(y_{2})$, i.e. 
\[
g_{a}(y_{2})\in((1-CH)g_{a}(y_{1})^{2},(1+CH)g_{a}(y_{1})^{2}).
\]

From Lemma \ref{local existence}, $g_{a}$ can be extended beyond
$y_{2}.$ And in a small interval $(y_{2},y_{2}+\delta)$, $g_{a}(y)>0,g_{a}'(y)>0,g_{a}''(y)>0.$
We want to show there is $y_{3}>y_{2}$ such that $g_{a}''(y_{3})=0$
and in $(y_{2},y_{3})$, $g_{a}''(y)>0.$ If it were false, then as
long as $g_{a}$ can be extended and $y>y_{2}$ we would have $g_{a}(y)>0,g_{a}'(y)>0,g_{a}''(y)>0.$
So when $y>y_{2}$, both $g_{a}$ and $g_{a}'$ would be monotonically
increasing. Suppose $g_{a}$ can be extended to $(y_{2},\tilde{y}_{3})$
but not until $\tilde{y}_{3}$, where $\tilde{y}_{3}$ is possibly
$+\infty$. However, as $k_{a}(y_{2})=g_{a}(y_{2})\leq(1+CH)g_{a}(y_{1})^{2}\leq(1+CH)C_{4}^{2}\leq\frac{1}{4}$,
$\{(g_{a},g_{a}');y>y_{2}\}$ can not intersect with $k_{a}=\frac{2}{3}$
or one would find a point $y_{3}$ such that $g_{a}''(y_{3})=0$ (where
$k_{a}\approx\frac{1}{2}$). When $y\rightarrow\tilde{y}_{3}^{-}$,
there is a limit for $(g_{a},g_{a}')$ which lies in the region bounded
by $g_{a}=g_{a}(y_{2}),g_{a}'=0$ and $k_{a}=\frac{2}{3}$. As in
the last lemma, it is easy to prove that $\tilde{y}_{3}$ has upper
bound and from Lemma \ref{closeness} and Lemma \ref{local existence},
we can extend the solution beyond $\tilde{y}_{3}$ which is a contradiction.
So there must be a $y_{3}>y_{2},y_{3}<+\infty$ such that $g_{a}''(y_{3})=0$.
We can choose $y_{3}$ as the minimum of such values, so in $(y_{2},y_{3})$,
$g_{a}(y)>0,g_{a}'(y)>0,g_{a}''(y)>0.$ 

From (\ref{tau w.r.t. g^2}) we know 
\[
|\tau(g_{a}(y_{2}),g_{a}'(y_{2}))-\tau(g_{a}(y_{3}),g_{a}'(y_{3}))|\leq CHg_{a}^{2}(y_{3}).
\]
So we have 
\[
g_{a}(y_{2})\in((1-CH)g_{a}(y_{3})^{2},(1+CH)g_{a}(y_{3})^{2}).
\]
Together with (\ref{tau(y2)}) we can deduce 
\[
g_{a}(y_{3})\in[(1-CH)g_{a}(y_{1}),(1+CH)g_{a}(y_{1})]
\]
 and 
\begin{equation}
\tau(g_{a}(y_{3}),g_{a}'(y_{3}))\in[(1-CH)\tau(g_{a}(y_{1}),g_{a}'(y_{1})),(1+CH)\tau(g_{a}(y_{1}),g_{a}'(y_{1}))].\label{tau(y3)}
\end{equation}

\end{proof}

\begin{lem} \label{y3 to y4} If (H1) or (H2) happens, $g_{a}$ can
be extended to $y_{4}>y_{3}$ and $g_{a}(y_{4})\in(\frac{1}{2},1.2),g_{a}'(y_{4})=0.$
In $[y_{3},y_{4}]$, $g_{a}$ is monotonically increasing and $g_{a}'$
is monotonically decreasing. There is $C>0$ such that 
\[
|\tau(g_{a}(y_{4}),g_{a}'(y_{4}))-\tau(g_{a}(y_{3}),g_{a}'(y_{3}))|\leq CH.
\]

\end{lem}

\begin{proof} 

From Lemma \ref{local existence}, $g_{a}$ can be extended beyond
$y_{3}$. As long as $g_{a}$ can be extended and $y>y_{3},g_{a}>0,g'_{a}>0$,
we are going to prove that $g''_{a}<0.$ We know in a small interval
$(y_{3},y_{3}+\delta)$ we have $g_{a}>0,g'_{a}>0$. From (\ref{Ka and ga})
we know, when $y>y_{3}$, as long as the solution could be extended
and $g_{a}>0,g_{a}'>0,k_{a}<\frac{2}{2+\hat{C}H}$, $k_{a}$ will
monotonically increase. If $k_{a}>\frac{1}{2-\hat{C}H}$, we have
\[
g''_{a}=\frac{1}{g}(1+g_{a}'^{2})(1-(2+\rho)k_{a})<0.
\]
We are going to prove that when $y>y_{3},g_{a}>0,g_{a}'>0,\frac{1}{2+\rho(y_{3})}\leq k_{a}\leq\frac{1}{2-\hat{C}H}$
we still have $g''_{a}<0.$ From (\ref{tau(y3)}) and $\frac{1}{2+\rho(y_{3})}\leq k_{a}\leq\frac{1}{2-\hat{C}H}$
we know $g_{a}<C_{4}$ and $g_{a}'\geq1$. Similarly, we can derive
(\ref{rho'}) and (\ref{(2+rho)ka}) again. So we have $g_{a}''(y)<0.$
So when $y>y_{3},g_{a}>0,g_{a}'>0$, we have $g_{a}$ is monotonically
increasing and $g_{a}'$ is monotonically decreasing. So in this interval,
$g_{a}$ can be regarded as a decreasing function of $g_{a}'$. We
claim that, when $y>y_{3}$ and $g_{a}>0,g_{a}'>0$, as long as the
solution could be extended, 
\begin{equation}
k_{a}\leq3.\label{ka upperbound}
\end{equation}
 To see this, by contradiction, we assume $g_{a}$ can be extended
to some $\tilde{y}_{4}>y_{3}$ such that when $y\in(y_{3},\tilde{y}_{4})$
we have $g_{a}(y)>0,g_{a}'(y)>0,$ and $k_{a}(\tilde{y}_{4})>3.$
Then there must be $y_{4}'\in(y_{3},\tilde{y}_{4})$ such that $k_{a}(y_{4}')=3.$
However if we integral (\ref{Ka and ga}) from $y_{4}'$ to $\tilde{y}_{4}$
we can get a contradiction. So we proved that $k_{a}\leq3.$ So we
have $g_{a}\leq3.$ From the monotonicity of $g_{a}$ and $g_{a}'$
and Lemma \ref{closeness}, obviously the solution can be extended
to some $y_{4}>y_{3}$ such that $g_{a}(y_{4})\leq3,g_{a}'(y_{4})=0.$ 

From (\ref{tau w.r.t. g^2}) we know 
\begin{equation}
|\tau(g_{a}(y_{4}),g_{a}'(y_{4}))-\tau(g_{a}(y_{3}),g_{a}'(y_{3}))|\leq CH.\label{tau(y4)-tau(y3)}
\end{equation}

\end{proof}

\begin{lem} \label{tau(y4) estimate}

If (H1) or (H2) happens, there holds 
\[
|\tau(g_{a}(y_{4}),g_{a}(y_{4}))-(a-a^{2})|+|g_{a}(y_{4})-a|\leq CH.
\]

Moreover, if we choose $a\in[0.95,1.05]$ and $H$ very small we have
$g_{a}(y_{4})\in[0.9,1.1]$.

\end{lem}

\begin{proof}From (\ref{tau(y1)-tau(0)})(\ref{tau(y3)})(\ref{tau(y4)-tau(y3)})
we get the conclusion easily.

\end{proof}

\begin{lem}\label{y4 to y5} If (H1) happens, $g_{a}$ can be extended
to $y_{5}>y_{4}$ such that $g_{a}(y_{5})>0,g_{a}'(y_{5})<0,g_{a}''(y_{5})=0$
and 
\[
|\tau(g_{a}(y_{5}),g_{a}'(y_{5}))-\tau(g_{a}(y_{4}),g_{a}'(y_{4}))|\leq CH.
\]
 If (H2) happens, $g_{a}$ can be extended to $y_{5}>y_{4}$ such
that 
\begin{align*}
\lim_{y\rightarrow y_{5}^{-}}g_{a} & \geq0,\\
\lim_{y\rightarrow y_{5}^{-}}g_{a}' & =-\infty.
\end{align*}
In either case, $g_{a}$ and $g_{a}'$ are monotonically decreasing
in $(y_{4},y_{5}).$

\end{lem}

\begin{proof} This lemma can be proved in the same way as Lemma \ref{H3 holds}.

\end{proof}

\section{Analysis of the singularities}

Define a subset $E(H)$ of $[0,95,1.05]$ which includes all the $a$
values such that (H1) happens to $g_{a}$. And denote $\bar{E}(H)=\{a\in[0.95,1.05]:a\notin E(H)\}.$
Obviously, we can choose $H$ sufficiently small such that $0.95\in E(H),1.05\in\bar{E}(H).$
And from the continuous dependence of the solution of ODE on its initial
values, $E(H)$ is an open subset of $[0,95,1.05].$ Let
\[
a'(H)=\inf\bar{E}(H).
\]
We have $a'(H)\in\bar{E}(H)$ and $a'(H)>0.95$ for small $H$. 

\begin{lem}\label{non compact}
\[
\lim_{a\rightarrow a'(H)^{-}}\max\{|g_{a}'(y_{1})|,|g_{a}'(y_{5})|\}=+\infty.
\]

\end{lem}

\begin{proof} If it were not true, we could find a sequence $a_{n}$
which tends to $a'(H)^{-}$ such that for any $n$ 
\[
\max\{|g_{a_{n}}'(y_{1})|,|g_{a_{n}}'(y_{5})|\}\leq C(H).
\]
 So $(g_{a_{n}},g_{a_{n}}')\in D(A_{1}(H),A_{2}(H),A_{3}(H),A_{4}(H))$
from Lemma \ref{y2 to y3}, \ref{y3 to y4}, \ref{tau(y4) estimate},
\ref{y4 to y5}. So 
\begin{equation}
\|g_{a_{n}}(y)\|_{C_{y}^{2}([0,y_{5}])}\leq C(H).\label{C2 estimate for g}
\end{equation}
Note that $y_{i},i=1,\cdots,5$ depend on $a$ and $H.$ Set $y_{0}=0$.
We are going to prove that there is $C(H)>0$ such that 
\begin{equation}
C(H)^{-1}\leq y_{i}-y_{i-1}\leq C(H).\label{yi-y(i-1) estimates}
\end{equation}
We only prove this for $y_{1}-y_{0}$ and $y_{2}-y_{1}$. The rest
inequalities follow similarly. It is easy to see that $k_{a}<3$ and
from the same argument as in the proof of Lemma \ref{H3 holds} we
see when $-\frac{1}{2}<g_{a}'\leq0$ on both sides of $y_{1}$, $k_{a}-\frac{1}{2}$
does not change sign and is bounded away from $0.$ By dividing the
orbit into $-\frac{1}{2}<g'_{a}\leq0$ and $g'_{a}\leq-\frac{1}{2}$,
we can apply Lemma \ref{beta-alpha estimate} to get (\ref{yi-y(i-1) estimates}).

By passing to a subsequence, we may assume that for fixed $H$, when
$a_{n}\rightarrow a'(H)^{-}$ we have $y_{i}\rightarrow y_{i}^{*}$
for $i=1,\cdots,5$ and 
\[
C(H)^{-1}\leq y_{i}^{*}-y_{i-1}^{*}\leq C(H).
\]
Differentiate the ODE (\ref{ODE with initial value}) with respect
to $y$ once. From (\ref{C2 estimate for g}) and (\ref{rho'}) we
know
\[
\|g_{a_{n}}\|_{C_{y}^{3}([0,y_{5}])}\leq C(H).
\]
So $g_{a_{n}}$ converges to some $\tilde{g}(y),y\in[0,y_{5}^{*}-\delta]$
in $C^{2}$ for any small $\delta>0.$ We know $\tilde{g}(y)$ solves
ODE (\ref{ODE with initial value}) with $a=a'(H).$ From the uniqueness
of the ODE, we know $g_{a'(H)}(y)=\tilde{g}(y),y\in[0,y_{5}^{*}).$
By uniform $C^{3}$ estimate of $g_{a_{n}}$, we know $g_{a'(H)}$
can be extended to $y=y_{5}^{*}$ and $g''_{a'(H)}(y_{5}^{*})=0$
and in each interval $[y_{i-1}^{*},y_{i}^{*}]$ it has the same monotonicity
as $g_{a}$ in $[y_{i-1},y_{i}].$ Then from the proof of Lemma \ref{y1 implies y2},
we know (H1) happens to $g_{a'(H)}$ which is a contradiction. So
we proved this lemma. 

\end{proof}

So there is $\delta_{1}(H)>0$ which depends on $H$ such that when
$-\delta_{1}(H)<a-a'(H)<0$, 
\[
\max\{|g_{a}'(y_{1})|,|g_{a}'(y_{5})|\}\geq\frac{1}{\sqrt{H}}.
\]

\begin{lem} \label{y5 estimate} There is $C>0$ which does not depend
on $a,H$, such that when $-\delta_{1}(H)<a-a'(H)<0$,
\[
\min\{|g_{a}'(y_{1})|,|g_{a}'(y_{5})|\}\geq\frac{C}{\sqrt{H}}.
\]

\end{lem}

\begin{proof}Let's first assume 
\[
\frac{1}{\sqrt{H}}\leq\max\{|g_{a}'(y_{1})|,|g_{a}'(y_{5})|\}=|g_{a}'(y_{1})|.
\]
From Lemma \ref{y2 to y3} and $g_{a}(y_{1})\sqrt{1+g_{a}'(y_{1})^{2}}\in[\frac{1}{2+\hat{C}H},\frac{1}{2-\hat{C}H}]$
we know there is a uniform $C>0$ such that 
\begin{equation}
\tau(g_{a}(y_{1}),g_{a}'(y_{1})),\tau(g_{a}(y_{3}),g_{a}'(y_{3}))\in[C^{-1}H,CH].\label{tau(y3) estimate}
\end{equation}
From Lemma \ref{y3 to y4} and Lemma \ref{y4 to y5}, we know 
\[
0<\tau(g_{a}(y_{5}),g_{a}'(y_{5}))\leq CH.
\]
Combining this with $g_{a}(y_{5})\sqrt{1+g_{a}'(y_{5})^{2}}\in[\frac{1}{2+\hat{C}H},\frac{1}{2-\hat{C}H}]$,
we know there is $C>0$ such that 
\[
|g_{a}'(y_{5})|\geq\frac{C}{\sqrt{H}}.
\]
If 
\[
\max\{|g_{a}'(y_{1})|,|g_{a}'(y_{5})|\}=|g_{a}'(y_{5})|
\]
we can prove the lemma similarly.

\end{proof}

Let $h(y)=\sqrt{1-(y-y_{4})^{2}},y\in[y_{4}-1,y_{4}+1].$ $h(y)$
solves the following ODE
\[
\begin{cases}
h''(y)-\frac{1}{h(y)}(1+h'(y)^{2})+2(1+h'(y)^{2})^{\frac{3}{2}}=0,\\
h(y_{4})=1,\\
h'(y_{4})=0.
\end{cases}
\]
There is exactly one $y_{6}\in(y_{3},y_{4})$ such that $g_{a}'(y_{6})=1.$
Define two continuous maps 
\begin{align*}
\Phi_{1}:[y_{6},y_{4}] & \rightarrow\mathbb{R},\\
\Phi_{2}:[y_{3},y_{6}] & \rightarrow\mathbb{R}
\end{align*}
 such that $g'_{a}(y)=h'(\Phi_{1}(y))$, $\Phi_{1}(y_{4})=y_{4}$
and $g_{a}(y)=h(\Phi_{2}(y))$, $\Phi_{2}(y_{6})\in(y_{4}-1,y_{4}).$
We use a graph below to illustrate the definition of $\Phi_{1},\Phi_{2}.$
In this graph, the horizontal direction represents $g_{a}'(y)$ or
$h'(y)$ and the vertical direction represents $g_{a}(y)$ or $h(y).$

\includegraphics[scale=0.48]{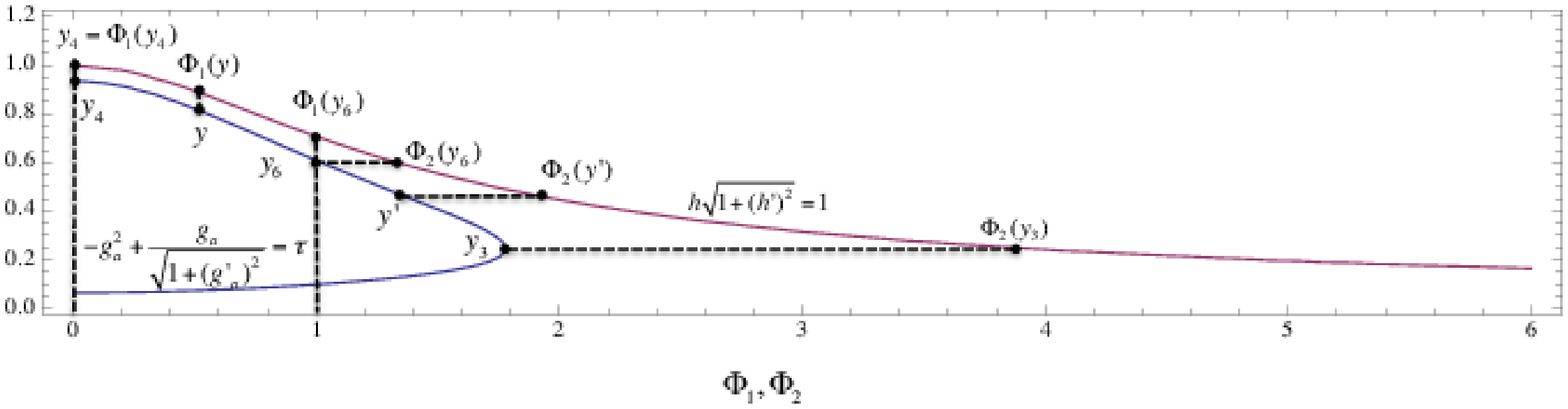}

We have the following estimates,

\begin{lem}\label{Phi1,Phi2} There is $C>0$ such that when $H$
is very small and $-\delta_{1}(H)<a-a'(H)<0$, 
\begin{align*}
|\Phi_{1}(y)-y| & \leq CH,\\
|\Phi_{2}(y)-y| & \leq CH(\frac{H}{g_{a}^{2}(y)}+|\log g_{a}(y)|).
\end{align*}

\end{lem}

\begin{proof}
\begin{align*}
\frac{d\Phi_{1}(y)}{dy} & =\frac{dg_{a}'(y)}{dy}\frac{d\Phi_{1}(y)}{dh'(\Phi_{1}(y))}\\
 & =\frac{\frac{1}{g_{a}}(1+g_{a}'^{2})-(2+\rho)(1+g_{a}'^{2})^{\frac{3}{2}}|_{y}}{\frac{1}{h}(1+h'^{2})-2(1+h'^{2})^{\frac{3}{2}}|_{\Phi_{1}(y)}}.
\end{align*}
 When $-\delta_{1}(H)<a-a'(H)<0$, $|\tau(g_{a}(y),g_{a}'(y))|\leq CH,y\in[0,y_{5}]$.
So we know when $y\in[y_{6},y_{4}]$, $g_{a}(y)\in[h(\Phi_{1}(y))(1-CH),h(\Phi_{1}(y))(1+CH)]$.
Hence, 
\[
|\frac{\frac{1}{g_{a}}(1+g_{a}'^{2})-(2+\rho)(1+g_{a}'^{2})^{\frac{3}{2}}|_{y}}{\frac{1}{h}(1+h'^{2})-2(1+h'^{2})^{\frac{3}{2}}|_{\Phi_{1}(y)}}-1|\leq CH.
\]
So the first estimate follows. For the second inequality, first we
notice 
\[
|h(\Phi_{1}(y_{6}))-h(\Phi_{2}(y_{6}))|=|h(\Phi_{1}(y_{6}))-g_{a}(y_{6})|\leq CH.
\]
From $h'(\Phi_{1}(y_{6}))=1,$ we know $\Phi_{1}(y_{6})=y_{4}-\frac{\sqrt{2}}{2}$.
So $h(\Phi_{1}(y_{6}))=\frac{\sqrt{2}}{2}$ and 
\[
|\Phi_{2}(y_{6})-\Phi_{1}(y_{6})|\leq CH.
\]
When $y\in[y_{3},y_{6}]$, 
\begin{align*}
y-y_{6} & =\int_{g_{a}(y_{6})}^{g_{a}(y)}\frac{1}{g_{a}'}dg_{a}\\
 & =\int_{g_{a}(y_{6})}^{g_{a}(y)}\frac{\tau+g_{a}^{2}}{\sqrt{g_{a}^{2}-(\tau+g_{a}^{2})^{2}}}dg_{a}
\end{align*}
and 
\begin{align*}
\Phi_{2}(y)-\Phi_{2}(y_{6}) & =\int_{h(\Phi_{2}(y_{6}))}^{h(\Phi_{2}(y))}\frac{1}{h'}dh\\
 & =\int_{g_{a}(y_{6})}^{g_{a}(y)}\frac{g_{a}}{\sqrt{1-g_{a}^{2}}}dg_{a}.
\end{align*}
So we have 
\begin{align*}
 & \Phi_{2}(y)-y-(\Phi_{2}(y_{6})-y_{6})\\
= & \int_{g_{a}(y_{6})}^{g_{a}(y)}(\frac{g_{a}}{\sqrt{1-g_{a}^{2}}}-\frac{\tau+g_{a}^{2}}{\sqrt{g_{a}^{2}-(\tau+g_{a}^{2})^{2}}})dg_{a}\\
= & \int_{g_{a}(y_{6})}^{g_{a}(y)}\frac{-\tau^{2}-2\tau g_{a}^{2}}{\sqrt{1-g_{a}^{2}}\sqrt{g_{a}^{2}-(\tau+g_{a}^{2})^{2}}(g_{a}\sqrt{g_{a}^{2}-(\tau+g_{a}^{2})^{2}}+(\tau+g_{a}^{2})\sqrt{1-g_{a}^{2}})}dg_{a}.
\end{align*}
When $y\in[y_{3},y_{6}]$, 
\begin{align*}
1-g_{a}^{2} & \geq1-(\frac{\sqrt{2}}{2}+CH)^{2}\geq\frac{1}{3},\\
|\tau| & \leq CH.
\end{align*}
From (\ref{ka upperbound}) we know when $y\in[y_{3},y_{4}]$
\[
\frac{\tau+g_{a}^{2}}{g_{a}^{2}}=\frac{1}{k_{a}}\geq\frac{1}{3}.
\]
So 
\begin{equation}
\tau+g_{a}^{2}\geq\frac{g_{a}^{2}}{3}.\label{tau+g^2  estimate}
\end{equation}
From $g_{a}\leq\frac{\sqrt{2}}{2}+CH,|\tau|\leq CH$, one can check
that, for some $C>0,$ 
\[
g_{a}^{2}-(\tau+g_{a}^{2})^{2}\geq Cg_{a}^{2}.
\]
So we have 
\[
|\frac{-\tau^{2}-2\tau g_{a}^{2}}{\sqrt{1-g_{a}^{2}}\sqrt{g_{a}^{2}-(\tau+g_{a}^{2})^{2}}(g_{a}\sqrt{g_{a}^{2}-(\tau+g_{a}^{2})^{2}}+(\tau+g_{a}^{2})\sqrt{1-g_{a}^{2}})}|\leq C\frac{H^{2}+Hg_{a}^{2}}{g_{a}^{3}}.
\]
Then
\begin{align*}
|\Phi_{2}(y)-y-(\Phi_{2}(y_{6})-y_{6})| & \leq C|\int_{g_{a}(y_{6})}^{g_{a}(y)}\frac{H^{2}+Hg_{a}^{2}}{g_{a}^{3}}dg_{a}|\\
 & =C(|\frac{H^{2}}{g_{a}^{2}(y)}-\frac{H^{2}}{g_{a}^{2}(y_{6})}|+H|\log g_{a}(y)-\log g_{a}(y_{6})|)\\
 & \leq CH(\frac{H}{g_{a}^{2}(y)}+|\log g_{a}(y)|).
\end{align*}
So we have when $y\in[y_{3},y_{6}]$ 
\begin{align*}
|\Phi_{2}(y)-y| & \leq|\Phi_{2}(y_{6})-\Phi_{1}(y_{6})|+|\Phi_{1}(y_{6})-y_{6}|+CH(\frac{H}{g_{a}^{2}}+|\log g_{a}|)\\
 & \leq CH(\frac{H}{g_{a}^{2}}+|\log g_{a}|+1).
\end{align*}
As $1\leq C|\log\frac{\sqrt{2}}{2}|$, we get the second estimate.

\end{proof}

\begin{lem}\label{y4 estimate} When $H$ is very small and $-\delta_{1}(H)<a-a'(H)<0$,
\begin{align*}
|y_{4}-2| & \leq C\sqrt{H},\\
|y_{5}-3| & \leq C\sqrt{H}.
\end{align*}

\end{lem}

\begin{proof}From (\ref{tau(y3)}) we know 
\[
h(\Phi_{2}(y_{3}))=g_{a}(y_{3})\leq C\sqrt{H}.
\]
So there is some $y_{6}'\in(y_{3,}y_{6})$ such that $g_{a}(y_{6}')=C\sqrt{H}$
where we use the same constant $C$ as the right hand side of the
above inequality. From Lemma \ref{Phi1,Phi2}, we know 
\[
|y_{6}'-\Phi_{2}(y_{6}')|\leq CH(\frac{H}{g_{a}^{2}(y_{6}')}+\log g_{a}(y_{6}'))\leq CH|\log H|.
\]
We know that $|\Phi_{2}(y_{6}')-(y_{4}-1)|\leq CH.$ So we have 
\[
|y_{6}'-(y_{4}-1)|\leq CH|\log H|.
\]
By applying Lemma \ref{beta-alpha estimate} we can get 
\[
|y_{6}'-y_{2}|\leq C\sqrt{H}.
\]

In the same way we can find $\hat{y}_{1}\in(1,y_{1})$ such that $g_{a}(\hat{y}_{1})=C\sqrt{H}$
for some $C>0$. And we can prove that 
\begin{align*}
|\hat{y}_{1}-1| & \leq CH\log H,\\
|\hat{y}_{1}-y_{2}| & \leq C\sqrt{H}.
\end{align*}

So 
\begin{align*}
|y_{4}-2| & =|y_{4}-1-y_{6}'+y_{6}'-y_{2}+y_{2}-\hat{y}_{1}+\hat{y}_{1}-1|\\
 & \le CH\log H+C\sqrt{H}\leq C\sqrt{H}.
\end{align*}
The second estimate follows similarly.

\end{proof}

\begin{lem}\label{y3-y4 integral} When $H$ is very small and $-\delta_{1}(H)<a-a'(H)<0$,
\begin{align*}
|\int_{y_{3}}^{y_{4}}\frac{1}{\sqrt{y^{2}+g_{a}(y)^{2}}}g_{a}g_{a}'dy-\int_{y_{4}-1}^{y_{4}}\frac{1}{\sqrt{y^{2}+h(y)^{2}}}hh'dy| & \leq CH,\\
|\int_{y_{3}}^{y_{4}}\frac{g_{a}-yg'_{a}}{(y^{2}+g_{a}(y)^{2})^{\frac{3}{2}}\sqrt{1+g_{a}'^{2}}}g_{a}g_{a}'dy-\int_{y_{4}-1}^{y_{4}}\frac{h-yh'}{(y^{2}+h(y)^{2})^{\frac{3}{2}}\sqrt{1+h'^{2}}}hh'dy| & \leq CH.
\end{align*}

\end{lem}

\begin{proof}For the first one we have an obvious reason to prove
\[
|\int_{y_{6}'}^{y_{4}}\frac{1}{\sqrt{y^{2}+g_{a}(y)^{2}}}g_{a}g_{a}'dy-\int_{\Phi_{2}(y_{6}')}^{y_{4}}\frac{1}{\sqrt{y^{2}+h(y)^{2}}}hh'dy|\leq CH
\]
instead, where $y_{6}'\in(y_{3},y_{6})$ and $g_{a}(y_{6}')=C\sqrt{H}.$
By using Lemma \ref{Phi1,Phi2}, this can be verified as 
\begin{align*}
 & |\int_{y_{6}'}^{y_{6}}\frac{1}{\sqrt{y^{2}+g_{a}(y)^{2}}}g_{a}g_{a}'dy-\int_{\Phi_{2}(y_{6}')}^{\Phi_{2}(y_{6})}\frac{1}{\sqrt{y^{2}+h(y)^{2}}}hh'dy|\\
\leq & \int_{g_{a}(y_{6}')}^{g_{a}(y_{6})}|\frac{g_{a}}{\sqrt{y^{2}+g_{a}(y)^{2}}}-\frac{g_{a}}{\sqrt{\Phi_{2}(y)^{2}+g_{a}(y)^{2}}}|dg_{a}\\
\leq & C\int_{g_{a}(y_{6}')}^{g_{a}(y_{6})}g_{a}H(\frac{H}{g_{a}^{2}}+|\log g_{a}|)dg_{a}\\
\leq & CH^{2}|\log H|+CH\leq CH,
\end{align*}
\[
|\int_{\Phi_{2}(y_{6})}^{\Phi_{1}(y_{6})}\frac{1}{\sqrt{y^{2}+h(y)^{2}}}hh'dy|\leq CH,
\]

\begin{align*}
 & |\int_{\Phi_{1}(y_{6})}^{y_{4}}\frac{1}{\sqrt{y^{2}+h(y)^{2}}}hh'dy-\int_{y_{6}}^{y_{4}}\frac{1}{\sqrt{y^{2}+g_{a}(y)^{2}}}g_{a}g_{a}'dy|\\
\leq & |\int_{y_{6}}^{y_{4}}\frac{1}{\sqrt{\Phi_{1}(y)^{2}+h(\Phi_{1}(y))^{2}}}h(\Phi_{1}(y))h'(\Phi_{1}(y))\Phi_{1}'(y)-\frac{1}{\sqrt{y^{2}+g_{a}(y)^{2}}}g_{a}g_{a}'dy|\\
\leq & CH.
\end{align*}

The proof of the second one is similar. 

\end{proof}

From Lemma \ref{y5 estimate}, using the same technique as the above
lemma, we can prove

\begin{lem}\label{y4-y5 integral}When $H$ is very small and $-\delta_{4}(H)<a-a'(H)<0,$
we have 
\begin{align*}
|\int_{y_{4}}^{y_{5}}\frac{1}{\sqrt{y^{2}+g_{a}(y)^{2}}}g_{a}g_{a}'dy-\int_{y_{4}}^{y_{4}+1}\frac{1}{\sqrt{y^{2}+h(y)^{2}}}hh'dy| & \leq CH,\\
|\int_{y_{4}}^{y_{5}}\frac{g_{a}-yg'_{a}}{(y^{2}+g_{a}(y)^{2})^{\frac{3}{2}}\sqrt{1+g_{a}'^{2}}}g_{a}g_{a}'dy-\int_{y_{4}}^{y_{4}+1}\frac{h-yh'}{(y^{2}+h(y)^{2})^{\frac{3}{2}}\sqrt{1+h'^{2}}}hh'dy| & \leq CH.
\end{align*}

\end{lem}

\begin{lem}\label{key lemma} We can fix $\lambda>0$ small and $p>0$
large such that there is $\delta(p,\lambda)>0$ small such that when
$0<H<\delta(p,\lambda)$ and $-\delta_{1}(H)<a-a'(H)<0$, we have
$|g_{a}'(y_{5})|\geq|g_{a}'(y_{1})|.$ And moreover 
\[
\lim_{a\rightarrow a'(H)^{-}}g_{a}'(y_{5})=-\infty,\liminf_{a\rightarrow a'(H)^{-}}g_{a}'(y_{1})>-\infty.
\]

\end{lem}

\begin{proof} From Lemma \ref{y3-y4 integral} and Lemma \ref{y4-y5 integral}
we know 
\begin{align*}
 & |\int_{y_{3}}^{y_{5}}(\frac{1}{\sqrt{y^{2}+g_{a}(y)^{2}}}+\frac{(g_{a}-yg'_{a})}{(y^{2}+g_{a}(y)^{2})^{\frac{3}{2}}\sqrt{1+g_{a}'^{2}}})g_{a}g_{a}'dy\\
 & -\int_{y_{4}-1}^{y_{4}+1}(\frac{1}{\sqrt{y^{2}+h(y)^{2}}}+\frac{h-yh'}{(y^{2}+h(y)^{2})^{\frac{3}{2}}\sqrt{1+h'^{2}}})hh'dy|\\
\leq & CH.
\end{align*}

The following facts are the key to our proof. When $y_{4}>1,$ 
\begin{align*}
\int_{y_{4}-1}^{y_{4}+1}(\frac{1}{\sqrt{y^{2}+h(y)^{2}}}+\frac{h-yh'}{(y^{2}+h(y)^{2})^{\frac{3}{2}}\sqrt{1+h'^{2}}})hh'dy & =0,\\
\int_{y_{4}-1}^{y_{4}+1}\frac{1}{\sqrt{y^{2}+h(y)^{2}}}hh'dy & =\frac{2}{3y_{4}^{2}},\\
\int_{y_{4}-1}^{y_{4}+1}\frac{h-yh'}{(y^{2}+h(y)^{2})^{\frac{3}{2}}\sqrt{1+h'^{2}}}hh'dy & =-\frac{2}{3y_{4}^{2}}.
\end{align*}

The proof of these facts takes direct calculations which we omit here.

From Lemma \ref{y4 estimate}, we know $|y_{4}-2|\leq C\sqrt{H}$.
So
\[
\int_{y_{4}-1}^{y_{4}+1}\frac{h-yh'}{(y^{2}+h(y)^{2})^{\frac{3}{2}}\sqrt{1+h'^{2}}}hh'dy=-\frac{1}{6}\pm O(\sqrt{H}).
\]
So we have 
\[
|\int_{y_{3}}^{y_{5}}(\frac{H}{\sqrt{y^{2}+g_{a}(y)^{2}}}+\frac{H(g_{a}-yg'_{a})}{(y^{2}+g_{a}(y)^{2})^{\frac{3}{2}}\sqrt{1+g_{a}'^{2}}})g_{a}g_{a}'dy|\leq CH^{2}.
\]
and when $p>3$ 
\begin{align*}
\int_{y_{3}}^{y_{5}}\frac{(p-3)H^{2}(g_{a}-yg'_{a})}{(y^{2}+g_{a}(y)^{2})^{\frac{3}{2}}\sqrt{1+g_{a}'^{2}}}g_{a}g_{a}'dy & \leq-(p-3)H^{2}(\frac{1}{6}+CH)\\
 & \leq-\frac{(p-3)}{7}H^{2},
\end{align*}
\begin{align*}
|\int_{y_{3}}^{y_{5}}\frac{(\phi-p)H^{2}(g_{a}-yg_{a}')}{(y^{2}+g_{a}(y)^{2})^{\frac{3}{2}}\sqrt{1+g_{a}'^{2}}}g_{a}g_{a}'dy| & \leq C|p-1|H^{2}\lambda^{2},
\end{align*}
From (\ref{phi' estimate})
\[
|\int_{y_{3}}^{y_{5}}(-\frac{H^{2}\phi'(y+g_{a}g'_{a})}{8d^{2}y^{2}\sqrt{1+g_{a}'^{2}}}+C(p,\lambda)O(H^{3}))g_{a}g_{a}'dy|\leq C|p-1|\lambda H^{2}+C(p,\lambda)\lambda^{2}H^{3}.
\]

Also note that 
\begin{align*}
|\int_{y_{1}}^{y_{3}}g_{a}g_{a}'\rho dy| & =|\int_{g_{a}(y_{1})}^{g_{a}(y_{2})}+\int_{g_{a}(y_{2})}^{g_{a}(y_{3})}\frac{1}{2}\rho dg_{a}^{2}|\\
 & \leq CH^{2}.
\end{align*}
So we can fix $\lambda>0$ small and $p>0$ large, and there is $\delta(\lambda,p)>0$
small that when $0<H<\delta(\lambda,p)$ 
\begin{align}
-CpH^{2}\leq\tau(g_{a}(y_{5}),g_{a}'(y_{5}))-\tau(g_{a}(y_{1}),g_{a}'(y_{1})) & =\int_{y_{1}}^{y_{5}}\rho g_{a}g_{a}'dy\leq-H^{2},\nonumber \\
-CpH^{2}\leq\tau(g_{a}(y_{5}),g_{a}'(y_{5}))-\tau(g_{a}(y_{3}),g_{a}'(y_{3})) & =\int_{y_{3}}^{y_{5}}\rho g_{a}g_{a}'dy\leq-H^{2}.\label{tau(y3) and tau(y5)}
\end{align}
From $\tau(g_{a}(y_{i}),g_{a}'(y_{i}))>0,i=1,3,5$, we know 
\[
\tau(g_{a}(y_{i}),g_{a}'(y_{i}))\geq H^{2},i=1,3.
\]

Note that $k_{a}(y_{i})=\frac{1}{2\pm O(H)},i=1,3,5.$ So 
\begin{align*}
g_{a}'(y_{i}) & =\frac{-\sqrt{\tau-4\tau^{2}}}{2\tau}=-\frac{1}{2\sqrt{\tau}}(1+O(\tau)),i=1,5,\\
g_{a}'(y_{3}) & =\frac{1}{2\sqrt{\tau}}(1+O(\tau)).
\end{align*}
So we have
\begin{align*}
0>g_{a}'(y_{1}) & \geq-\frac{C}{H},\\
0<g_{a}'(y_{3}) & \leq\frac{C}{H},
\end{align*}
and 
\[
\lim_{a\rightarrow a'(H)^{-}}g_{a}'(y_{5})=-\infty.
\]

\end{proof}

Now we can analyze the singular limit of $g_{a}$ as $a\rightarrow a'(H)^{-}.$

\begin{lem}\label{estimates at singular point} Choose $p,\lambda,\delta(p,\lambda),\delta_{1}(H)$
such that the conclusion of Lemma \ref{key lemma} holds. When $a\rightarrow a'(H)^{-}$,
by passing to a sequence, we have $y_{i}\rightarrow y_{i}^{*},i=1,\cdots,5$
and $g_{a}$ converges to $g_{a'(H)}$ in $C^{2}$ sense on every
$[0,\tilde{y}]\subset[0,y_{5}^{*})$. (H2) happens to $g_{a'(H)}(y)$
and 
\begin{align*}
\lim_{y\rightarrow y_{5}^{*-}}g_{a'(H)}(y) & =0,\\
\lim_{y\rightarrow y_{5}^{*-}}g_{a'(H)}'(y) & =-\infty.
\end{align*}

\end{lem}

\begin{proof}

From 
\[
\lim_{a\rightarrow a'(H)^{-}}g_{a}'(y_{5})=-\infty
\]
and $g_{a}'$ is monotonically decreasing on $[y_{4},y_{5}]$ we know
for any $j\in\mathbb{N}^{+}$ when $a$ is sufficiently close to $a'(H)$,
there is exactly one $y_{5}^{j}\in(y_{4},y_{5})$ such that $g_{a}'(y_{5}^{j})=-j.$
From 
\[
-\frac{C}{H}\leq g_{a}'(y_{1})<0<g_{a}'(y_{3})\leq\frac{C}{H}
\]
we know that $g_{a}'$ is bounded independent of $a$ on $[0,y_{5}^{j}]$.
From the analysis of the last section, it is not hard to prove that
$g_{a}$ has both positive upper bound and positive lower bound on
$[0,y_{5}^{j}]$. So $(g_{a},g_{a}')\in D(A_{1},A_{2},A_{3},A_{4}),$
where $A_{i}=A_{i}(H,j),i=1,\cdots,4.$ So on $[0,y_{5}^{j}]$, $g_{a}$
has uniform  $C^{3}$ bounds. By using the same argument as in Lemma
\ref{non compact}, we can prove that there is $C(p,\lambda,H)>0$
such that 
\[
C(p,\lambda,H)^{-1}\leq y_{i}-y_{i-1}\leq C(p,\lambda,H),i=1,\cdots,4.
\]
And from Lemma \ref{beta-alpha estimate}, we can prove that $C^{-1}<y_{5}^{j}-y_{4}<C.$
So by passing to a sequence, we can assume $y_{i}\rightarrow y_{i}^{*},i=1,\cdots,4,y_{5}^{j}\rightarrow y_{5}^{j*}.$
And $g_{a}$ converges to $g_{a'(H)}$ in $C^{2}([0,y_{5}^{j*}])$
sense, where $g_{a'(H)}$ solves (\ref{ODE with initial value}) with
initial value $a'(H)$. By a diagonal method, we can choose $a_{i}\rightarrow a'(H)^{-}$
such that for any $j\in\mathbb{N}^{+}$ 
\[
\lim_{i\rightarrow+\infty}y_{5}^{j}=y_{5}^{j*}.
\]
It is obvious that $y_{5}^{j*}$ is monotonically increasing in $j$.
Note that we have 
\begin{align*}
y_{5}^{j+1}-y_{5}^{j} & =\int_{-(j+1)}^{-j}\frac{g_{a}dg_{a}'}{(1+g_{a}'^{2})((2+\rho)k_{a}-1)}.
\end{align*}
For any $j>0$, when $-j-1<g_{a}'<-j$, we can choose $a$ close to
$a'(H)$ such that $k_{a}$ is bounded away from (and bigger than)
$\frac{1}{2}$ . So there is $C_{j}>0$ such that 
\[
y_{5}^{j+1}-y_{5}^{j}\geq C_{j}.
\]
So we know $|y_{5}^{j+1*}-y_{5}^{j*}|\geq C_{j}.$ From the fact that
the second derivative of $g_{a}$ has uniform bound in $[0,y_{5}^{j}]$,
we know
\[
g_{a'(H)}'(y_{5}^{j*})=\lim_{a\rightarrow a'(H)^{-}}g_{a}'(y_{5}^{j*})=\lim_{a\rightarrow a'(H)^{-}}g_{a}'(y_{5}^{j})=-j.
\]
So we know 
\[
\lim_{j\rightarrow+\infty}g_{a'(H)}^{'}(y_{5}^{j*})=-\infty.
\]

Note that 
\[
|y_{5}^{j}-y_{5}|\leq\frac{C}{j}.
\]
If we take a subsequence further, we have $y_{5}\rightarrow y_{5}^{*}.$
So we have $|y_{5}^{j*}-y_{5}^{*}|\leq\frac{C}{j}.$ So we have 
\[
\lim_{j\rightarrow+\infty}y_{5}^{j*}=y_{5}^{*}.
\]
From the monotonicity of $g_{a}'$ we know 
\[
\lim_{y\rightarrow y_{5}^{*}}g'_{a'(H)}(y)=-\infty.
\]

For $k_{a}=g_{a}\sqrt{1+g_{a}'^{2}}$, by analyzing (\ref{Ka and ga}),
we know when $y\in[y_{4},y_{5}]$, 
\begin{equation}
\frac{1}{2+CH}\le k_{a}\leq\frac{1}{1-CH}.\label{ka estimate on =00005By4,y5=00005D}
\end{equation}
So we have 
\[
0<g_{a}(y_{5}^{j})\leq\frac{C}{j}
\]
hence $g_{a'(H)}(y_{5}^{j*})\leq\frac{C}{j}$ and 
\[
\lim_{y\rightarrow y_{5}^{*}}g_{a'(H)}(y)=0.
\]

At last, noticing that any $[0,\tilde{y}]\subset[0,y_{5}^{*})$ is
contained in some $[0,y_{5}^{j})$ for any $a$ close to $a'(H).$
So we can prove the uniform convergence on $[0,\tilde{y}].$ Obviously,
(H2) happens to $g_{a'(H)}$.

\end{proof}

Here we draw a graph of $g_{a'(H)}(y),y\in[0,y_{5}^{*}).$

\includegraphics[scale=0.4]{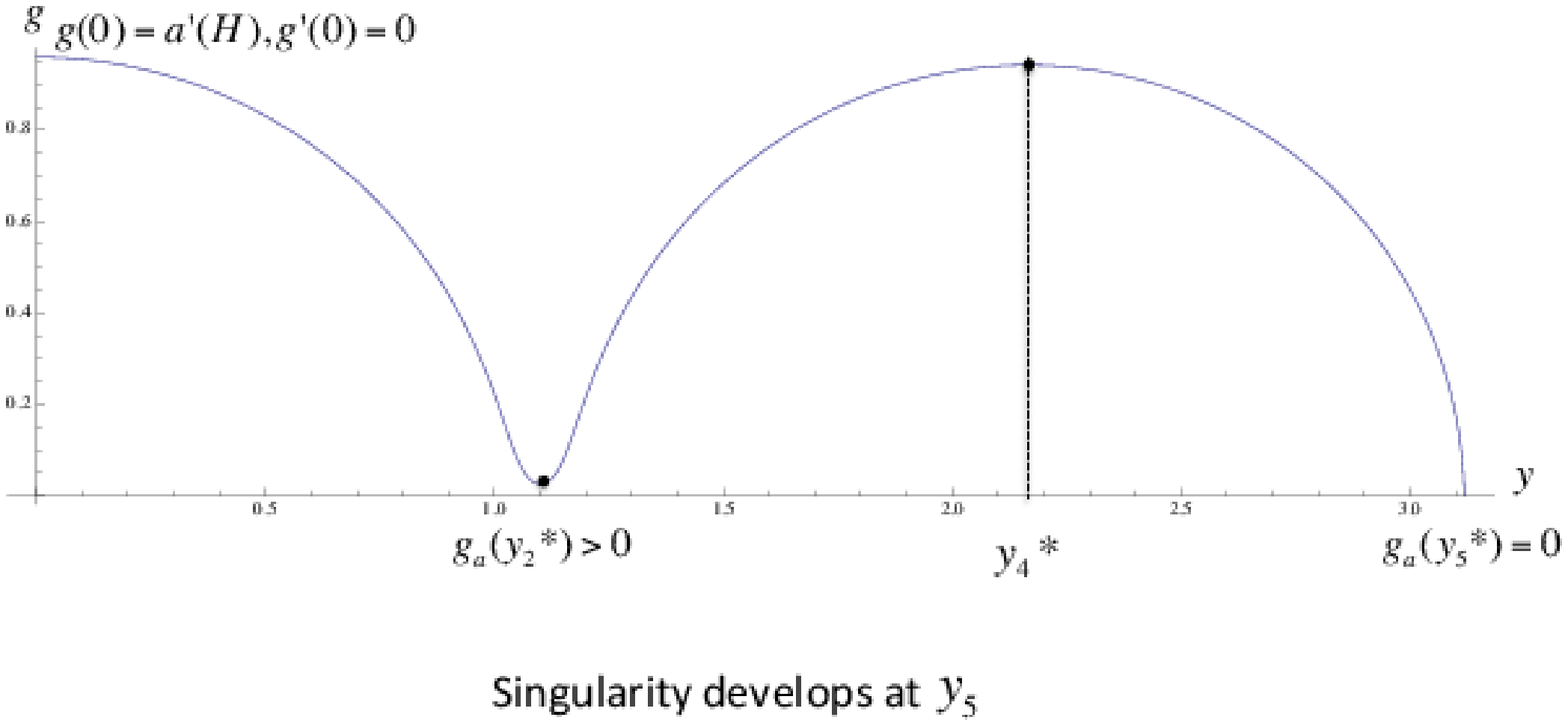}

$g_{a'(H)}(y)$ is monotonically decreasing in $[y_{4}^{*},y_{5}^{*})$.
Now we use $g$ to represent $g_{a'(H)}$ for short. Suppose $y=y(g),g\in(0,g_{a'(H)}(y_{4}^{*})]$
is the inverse function of $g_{a'(H)}(y),y\in[y_{4}^{*},y_{5}^{*}).$
We make even extension to $y(g)$. 
\[
y^{even}(g)=\begin{cases}
y(g), & g\in(0,g_{a'(H)}(y_{4}^{*})],\\
y_{5}^{*}, & g=0,\\
y(-g), & g\in[-g_{a'(H)}(y_{4}^{*}),0).
\end{cases}
\]
Then we have

\begin{lem}\label{regularity at singular point}$y^{even}(g)$ has
continuous second derivative in $(-\delta,\delta)$ for small $\delta.$

\end{lem}

\begin{proof} $y^{even}(g)$ is continuous at $0.$ When $g>0,$
\[
\frac{dy^{even}}{dg}=\frac{1}{g'}.
\]
From (\ref{ka estimate on =00005By4,y5=00005D}), we know when $g$
is close to $0$, $g'\sim\frac{1}{g}.$ So we have 
\[
\lim_{g\rightarrow0^{+}}\frac{dy^{even}}{dg}=0.
\]
So $y^{even}(g)$ has $0$ derivative at $g=0$. For the second derivative,
we note that 
\begin{align}
\frac{d^{2}y^{even}}{dg^{2}} & =\frac{d}{dg}(\frac{1}{g'})=\frac{-g''}{g'^{3}}\nonumber \\
 & =-\frac{1+g'^{2}}{g'^{2}}\frac{1}{g}\frac{1}{g'}+(2+\rho)\frac{(1+g'^{2})^{\frac{3}{2}}}{g'^{3}}.\label{y(even)-second}
\end{align}
Regard $\frac{1+g'^{2}}{g'^{2}}$ as $G_{1}(g)$ and $\frac{(1+g'^{2})^{\frac{3}{2}}}{g'^{3}}$
as $G_{2}(g)$. We get a linear ODE of first order,
\[
\begin{cases}
\frac{d}{dg}(\frac{1}{g'})+\frac{G_{1}(g)}{g}\frac{1}{g'} & =(2+\rho)G_{2}(g),\\
\frac{1}{g'}(0)=0,
\end{cases}
\]
with 
\begin{align*}
\lim_{g\rightarrow0^{+}}G_{1}(g) & =1,\\
G_{1}(g) & =1+O(g^{2}),\\
\lim_{g\rightarrow0^{+}}G_{2}(g) & =1.
\end{align*}
The solution is 
\[
\frac{1}{g'}=\frac{\int_{0}^{g}(2+\rho)G_{2}(s)s\exp(O(s^{2}))ds}{g\exp(O(g^{2}))}.
\]
So 
\begin{align*}
\lim_{g\rightarrow0^{+}}\frac{1}{gg'} & =\lim_{g\rightarrow0^{+}}\frac{\int_{0}^{g}(2+\rho)G_{2}(s)s\exp(O(s^{2}))ds}{g^{2}\exp(O(g^{2}))}\\
 & =\lim_{g\rightarrow0^{+}}\frac{(2+\rho)G_{2}(g)g\exp(O(g^{2}))}{2g\exp(O(g^{2}))+g^{2}\exp(O(g^{2}))O(g)}\\
 & =1+\frac{1}{2}\lim_{g\rightarrow0^{+}}\rho.
\end{align*}
From (\ref{rho-original}) we know as $y\rightarrow y_{5}^{*},g_{a'(H)}\rightarrow0^{+},g_{a'(H)}'\rightarrow-\infty$,
we have $d\rightarrow y_{5}^{*},\frac{g_{a}-yg_{a}'}{\sqrt{1+g_{a}'^{2}}}\rightarrow y_{5}^{*}.$
So the limit 
\[
\lim_{g\rightarrow0^{+}}\rho
\]
exists. So from (\ref{y(even)-second}), $\frac{d^{2}y^{even}}{dg^{2}}(0)$
exists. 

\end{proof}

\section{Proof of the main theorems}

Now we prove Theorem \ref{thm 1}. For any $H$ sufficiently small,
we have got $g_{a'(H)}(y)$, where we have assume that $g_{a'(H)}(y_{5}^{*})=0$.
Consider the surfaces of revolution $\Sigma^{+}(H)$ defined by 
\[
r=\frac{2}{H}g_{a'(H)}(\frac{H}{2}x),x\in[0,\frac{2y_{5}^{*}}{H}].
\]
 From our construction, when $x\in[0,\frac{2y_{5}^{*}}{H})$ the surface
has constant mean curvature $H$. When $x=\frac{2y_{5}^{*}}{H}$,
from Lemma \ref{regularity at singular point}, the surface actually
has well defined and continuous mean curvature at this single point.
So the mean curvature at this point must be $H.$ Let $\Sigma^{-}(H)$
be defined by 
\[
r=\frac{2}{H}g_{a'(H)}(-\frac{H}{2}x),x\in[-\frac{2y_{5}^{*}}{H},0).
\]
And let $\Sigma(H)=\Sigma^{+}(H)\cup\Sigma^{-}(H).$ Then $\Sigma(H)$
has constant mean curvature $H$ globally. It is obvious that $\Sigma(H)$
has sphere topology and it is embedded.

From Lemma \ref{d lower bound}, we know $C^{-1}H^{-1}\leq l_{0}\leq CH^{-1}.$
So it follows that for any compact set $K$, as long as $H$ is sufficiently
small, $\Sigma(H)$ separates $K$ from infinity. 

Now we calculate $|\Sigma(H)|$. Let $|\Sigma(H)|_{e}$ denote the
area in Euclidean metric. We know that 
\[
|\frac{|\Sigma(H)|_{e}}{|\Sigma(H)|}-1|\leq Cl_{0}^{-1}\leq CH.
\]
So we need only to prove that $|H^{2}|\Sigma(H)|_{e}-48\pi|\leq C(p)H.$
If we denote $g_{a'(H)}$ as $g$ for short, we have 
\[
H^{2}|\Sigma(H)|_{e}=16\pi\int_{0}^{y_{5}^{*}}g\sqrt{1+g'^{2}}dy.
\]
First we consider the integral on $[y_{2}^{*},y_{4}^{*}].$ From (\ref{tau(y3) and tau(y5)})
we know, 
\[
H^{2}\leq\tau(g(y_{3}^{*}),g'(y_{3}^{*}))\leq CpH^{2}
\]
Then for some $C(p)>0$, $CH\leq g(y_{3}^{*})\leq C(p)H,C(p)H^{-1}\leq g'(y_{3}^{*})\leq CH^{-1}$.
There is $y_{6}^{*}\in[y_{3}^{*},y_{4}^{*}]$ such that $g'(y_{6}^{*})=1$.
We have 
\begin{align*}
|\int_{y_{6}^{*}}^{y_{4}^{*}}g\sqrt{1+g'^{2}}dy-\int_{\Phi_{1}(y_{6}^{*})}^{y_{4}^{*}}h(s)\sqrt{1+h'^{2}(s)}ds| & \leq C(p)H,\\
|\int_{\Phi_{2}(y_{6}^{*})}^{\Phi_{1}(y_{6}^{*})}h(s)\sqrt{1+h'^{2}(s)}ds| & \leq C(p)H,\\
|\int_{g(y_{3}^{*})}^{g(y_{6}^{*})}g\frac{\sqrt{1+g'^{2}}}{g'}dg-\int_{g(y_{3}^{*})}^{g(y_{6}^{*})}h\frac{\sqrt{1+h'^{2}}}{h'}dh| & \leq C(p)H,\\
|\int_{y_{2}^{*}}^{y_{3}^{*}}g\sqrt{1+g'^{2}}dy|,|\int_{y_{4}^{*}-1}^{\Phi_{2}(y_{3}^{*})}h\sqrt{1+h'^{2}}dy| & \leq C(p)H,
\end{align*}
where the third one holds because 
\[
|g(\frac{\sqrt{1+g'^{2}}}{g'}-\frac{\sqrt{1+h'^{2}}}{h'})|\leq\frac{C(p)H^{2}}{g}+C(p)gH
\]
and the fourth one holds because $y_{3}^{*}-y_{2}^{*}\leq C(p)H.$ 

By direct calculation we know 
\[
\int_{y_{4}^{*}-1}^{y_{4}^{*}}h\sqrt{1+h'^{2}}dy=1.
\]
So gathering all the inequalities we have 
\[
|\int_{y_{2}^{*}}^{y_{4}^{*}}g\sqrt{1+g'^{2}}dy-1|\leq C(p)H.
\]
In a similar way we can prove that 
\begin{align*}
|\int_{0}^{y_{2}^{*}}g\sqrt{1+g'^{2}}dy-1| & \leq C(p)H,\\
|\int_{y_{4}^{*}}^{y_{5}^{*}}g\sqrt{1+g'^{2}}-1| & \leq CH.
\end{align*}
So we know 
\[
|\int_{0}^{y_{5}^{*}}g\sqrt{1+g'^{2}}-3|\leq C(p)H.
\]
So 
\[
|H^{2}|\Sigma(H)|_{e}-48\pi|\leq C(p)H.
\]
However, if we revise Qing and Tian's proof of the uniqueness CMC
spheres in \cite{Qing-Tian-CMC}, we know for a stable CMC sphere
that separates the compact part from infinity and with $l_{0}$ large,
$H^{2}|\Sigma|$ should be close to $16\pi.$ So the CMC spheres we
constructed are unstable.

Now we prove Theorem \ref{thm 2}. Let's revise Lemma \ref{key lemma}.
Similarly we can choose $p<0$ small and $\delta(p)>0$ such that
when $0<H<\delta(p)$ and $-\delta_{1}(H)<a-a'(H)<0$, we have $|g_{a}'(y_{5})|<|g_{a}'(y_{1})|.$
So we can choose a proper $p=p(H)$ such that $|g_{a}'(y_{5})|=|g_{a}'(y_{1})|.$
In this case, we have 
\[
\lim_{a\rightarrow a'(H)^{-}}g_{a}'(y_{5})=\lim_{a\rightarrow a'(H)^{-}}g_{a}'(y_{1})=-\infty.
\]
Now we can analyze the behavior of $g_{a}$ on $y_{1}$ and $y_{3}$
by using the same method as used in Lemma \ref{estimates at singular point}.
We will get a singular limit of three spheres (each one is embedded).
They share the same axis and the central one meets its two neighbors
at two poles where $x_{2}=x_{3}=0$. By using the same method as used
in Lemma \ref{regularity at singular point}, we know all the three
spheres are smooth and have constant mean curvature $H.$ 

The problem is, for different $H$, we may get different $p(H)$ (which
is bounded independent of $H$). Note that for a particular $H>0$,
when we do all the constructions above, only the metric in a domain
\[
D(H)=\{(x_{1},x_{2},x_{3})|C^{-1}H^{-1}\leq\sqrt{x_{1}^{2}+x_{2}^{2}+x_{3}^{2}}\leq CH^{-1}\}
\]
 really matters. If we have constructed $\Sigma_{n}^{i},i=1,2,3$
for some $H_{n}>0$, we can choose a much larger scale to construct
$\Sigma_{n+1}^{i},i=1,2,3.$ We may assume that $D(H_{i}),D(H_{j})$
are disjoint for $i\neq j$. We choose $p$ as a smooth function of
$l=\sqrt{x_{1}^{2}+x_{2}^{2}+x_{3}^{2}}$ such that $p=p_{i}$ in
each $D(H_{i}).$ We can assume that $D(H_{i})$ is far away from
$D(H_{i+1})$ and $p'(l)$ is small enough such that the $\phi_{\lambda,p}$
term does not influence the mass. In this way, we get a smooth asymptotically
Schwarzschild metric with mass $1$. Then we can prove Theorem \ref{thm 2}.

\bibliographystyle{plain}
\bibliography{/Users/apple/Documents/Documents/Math/My-Papers/Mathbib}

\end{document}